\documentclass[final]{amsart}

\usepackage{multirow}
\usepackage[centering,text={15.5cm,22cm},
		marginparwidth=20mm]{geometry}

\usepackage{amsmath}
\usepackage{amssymb}
\usepackage{graphicx}
\usepackage[all]{xy}
\usepackage{fontenc}
\usepackage{amsthm}
\usepackage{enumerate}
\theoremstyle{definition}
\newtheorem{theorem}{Theorem}
\newtheorem*{theorema}{Theorem}
\newtheorem{corollary}[theorem]{Corollary}
\newtheorem{lemma}[theorem]{Lemma}
\newtheorem{proposition}[theorem]{Proposition}

\newtheorem{definition}[theorem]{Definition}
\newtheorem{example}[theorem]{Example}
\newtheorem{remark}[theorem]{Remark}
\numberwithin{equation}{section}
\newcommand{\C} {\mathbb{C}}
\newcommand{\CS} {\mathbb{C}^{*}}
\newcommand{\PP} {\mathbb{P}}
\newcommand{\PO} {\mathbb{P}^{1}}
\newcommand{\AO} {\mathbb{A}^{1}}

\begin{document}

\title{Relative Hilbert scheme of points}

\author{Iman \textsc{Setayesh}}
\address{Department of Mathematics, Faculty of Mathematical Sciences, Tarbiat Modares University, P.O. Box 14115-137, Tehran, Iran.}
\email{setayesh@modares.ac.ir}



\begin{abstract}
 Let $D$ be a smooth divisor on a non singular surface $S$. We compute the Betti numbers of the Hilbert scheme of points of $S$ relative to $D$. In the case of $\PP^2$ and a line in it, we give an explicit set of generators and relations for the corresponding cohomology groups. 
\end{abstract}

\maketitle


\section{Introduction}

\subsection{Hilbert Scheme}

Let $X$ be a projective scheme over complex numbers, $\mathcal{L}$ be an ample line bundle on it and $P$ be a polynomial. Consider the contravariant functor

\begin{center}
$\mathcal{H}ilb^{P}_{X}$ $:$ $Sch \longrightarrow Sets$
\end{center}
from the category of schemes to sets, which is given by:
\begin{displaymath}
\mathcal{H}ilb^{P}_{X}(T)=\left\{
 \begin{array}{c|c}
    \xymatrix{Z \ar@{^{(}->}[rr]^i \ar[dr]_{\pi} & & X \times T \ar[dl]^{p_2} \\
                                       & T &  }  & \begin{array}{l}
                                       Z \text{ is a closed subscheme of } X \times T \\
                                       \pi\text{ is flat} \\
                                       Z_t \text{ has Hilbert polynomial}\\
                                       \text{       equal to $P$, for all } t \in T
                                       \end{array}
\end{array}
\right\}
\end{displaymath}

Since $\pi$ is flat all $Z_t$'s have the same Hilbert polynomial and the definition makes sense. 

\begin{theorema} \cite{Grothendeick} 
The functor $\mathcal{H}ilb^{P}_{X}$ is representable by a scheme. 
\end{theorema}

The proof is due to Grothendieck \cite{Grothendeick} with simplifications by Mumford \cite{Mumford}. The idea is that a subscheme $Z\subseteq \PP^{n}$ is given by its equations, which gives an injection of sets:

\begin{center}
$\lbrace$ subschemes of $\PP^{n}\rbrace\hookrightarrow \lbrace$ linear subspaces of $\mathbb{C}[x_0,...,x_{n}]\rbrace$.
\end{center}

The main technical point is to show that the infinite dimensional Grassmannian can be replaced by a finite dimensional one, and the image of the left hand side will be an algebraic subvariety of it. More details on the construction of Hilbert schemes and their infinitesimal properties can be found in \cite{Kollar}.

For the constant polynomial $n$ the associated Hilbert scheme parametrizes the set of all subschemes of $X$ with zero dimensional support and length $n$, which is called the Hilbert scheme of $n$ points on $X$ and is denoted by $X^{[n]}$.

Let $S$ be a quasi-projective non-singular surface, and $D$ be a smooth Cartier divisor. The Hilbert scheme of points of $S\setminus D$, which is not proper, may be compactified relative to $D$. This space has been constructed in \cite{LiWu}. We have an informal description of the relative Hilbert scheme of points in subsection \ref{informal} followed by a more precise definition in section \ref{background}.

The relative Hilbert scheme and more generally the moduli spaces of stable ideal sheaves play an important role in the study of degenerations of the moduli space of ideal sheaves in Donaldson-Thomas theory. The moduli stack of stable ideal sheaves shows up naturally in the study of degenerations of moduli spaces.

Consider a degeneration of a smooth variety $X^{sm}$ into a union of two smooth irreducible varieties $X_0=X_1\cup X_2$ intersecting transversally along a smooth divisor. One can consider the moduli space of stable sheaves on $X_0$. Unfortunately the standard tangent-obstruction theory of this problem is not perfect in general, and the existence of virtual cycle is not known. To overcome this problem we replace $X_0$ by the stack of expanded degenerations $\mathfrak{X}$ introduced by Jun Li \cite{Li}, and construct the moduli stack of stable ideal sheaves on $\mathfrak{X}$. It is described in more details in section \ref{background}

\subsection{G\"{o}ttsche's Formula}

For a smooth scheme $Y$ with $dim_{\C}(Y)=n$ the Poincar\'{e} polynomial $\displaystyle P_Y(t)$ and the normalized Poincar\'{e} polynomial $\hat{P}_Y(t)$ are defined by 

\begin{center}
$\displaystyle P_Y(t)=\sum_{i=0}^{2n}b_i(Y)t^i$ \hspace{5 mm} and \hspace{5 mm} $\hat{P}_Y(t)=\sum_{i=0}^{2n}b_i(Y)t^{i-n}$
\end{center}

\noindent respectively, where $b_i(Y)$ is the $i^{th}$ Betti number of $Y$. For singular varieties we use the virtual Poincar\'{e} polynomial given by the Hodge filtration as follows (for more details see section \ref{vir-betti}). $$P_{X}=\displaystyle \sum_{i,j} (-1)^{i+j} dim ( gr_{W}^{j}H^i_{c}(X))t^{j} .$$ 

The Poincar\'{e} polynomial of the Hilbert scheme of points was first studied by Ellingsrud and Str{\o}mme in \cite{EllingsrudStromme}, where they calculated the Poincar\'{e} polynomial of $(\mathbb{C}^{2})^{[n]}$. They used the Bialynicki-Birula decomposition associated with the natural torus action on $(\mathbb{C}^{2})^{[n]}$ to obtain an algebraic cell decomposition. Extending the result of G\"{o}ttsche \cite{Gottsche1} for the projective surfaces, G\"{o}ttsche and Soregel \cite{GS} proved the following theorem for quasi-projective non-singular surfaces. Their method uses Borho-MacPherson's formula for the direct image of the intersection cohomology. 

\begin{theorema}\cite{GS} 

Let $X$ be a quasi-projective nonsingular surface. Then the generating function for the Poincar\'{e} polynomial of the Hilbert scheme of $n$ points on $X$ is given by:

\begin{center}

$\displaystyle\sum_{n=0}^{\infty}
q^{n}P_{X^{[n]}}(t)=\prod_{m=1}^{\infty}\frac{(1+t^{2m-1}q^{m})^{b_{1}(X)}(1+t^{2m+1}q^{m})^{b_{3}(X)}}{(1-t^{2m-2}q^{m})^{b_{0}(X)} (1-t^{2m}q^{m})^{b_{2}(X)}(1-t^{2m+2}q^{m})^{b_{4}(X)}}$.

\end{center}

\end{theorema}

 In this paper we prove:

\begin{theorem}
The generating function for the normalized Poincar\'{e} polynomial of the relative Hilbert scheme of points is given by:

$$
\sum q^{n}\hat{P}_{S^{[n]}_{D}}(t) = \frac{(t^{2}-1)\hat{H}_{S}(q,t)}{t^{2}C_{D}(q,t)-C_{D}(q,t^{-1})},
$$

\noindent where $\hat{H}_{S}(q,t):=\sum q^{n}\hat{P}_{S^{[n]}}(t)$ is the normalized Poincar\'{e} polynomial of the Hilbert scheme of points on $S$ and 

$$C_{D}(q,t)=\displaystyle\prod_{m=1}^{\infty}\frac{(1+t^{-1}q^{m})^{b_{1}(D)}}{(1-t^{-2} q^{m})^{b_{0}(D)}(1-q^{m})^{b_{2}(D)}}.$$

\end{theorem}

\subsection{Nakajima's Basis}

Let $X$ be a quasi-projective non-singular surface. There is a natural map from the Hilbert scheme of $n$ points in $X$ to the $n^{th}$-symmetric product $X^{(n)}$ of $X$ given by:

\begin{center}
$\rho : X^{[n]} \to X^{(n)}, Z \mapsto \displaystyle\sum \ell(\mathcal{O}_{Z,p})p$.
\end{center}

For $i>0$ we define the cycles $P[i]\subset \displaystyle\coprod_{n} X^{[n-i]} \times X^{[n]} \times X$ to be:

\begin{center}
$P[i]:=\displaystyle \coprod_{n} \{ (\mathcal{I}_1,\mathcal{I}_2,p)\in X^{[n-i]} \times X^{[n]} \times X \hspace{2mm}|\hspace{2mm}  \mathcal{I}_1\supset\mathcal{I}_2  , \hspace{2mm} \rho(\mathcal{I}_2)-\rho(\mathcal{I}_1)=n[p]\} .$
\end{center}

For $i<0$ we define $P[i]$ by interchanging the role of $\mathcal{I}_1$ and $\mathcal{I}_2$ in the above two conditions. 

Let $H_{*}^{lf}(X)$ be the Borel-Moore Homology of X. For $\alpha \in H_{*}^{lf}(X)$ and $\beta \in H_{*}(X)$ and $i>0$ we define the operators  $P_{\alpha}[i]$ and $P_{\beta}[-i]$ by\vspace{3 mm}

 \hspace{15mm} $P_{\alpha}[i]: H_{*}(X^{[n]}) \to H_{*}(X^{[n-i]})$

 \hspace{42mm}$\gamma \mapsto p_{2*}((p_1^*\gamma)\cap( \pi_{1,2*}(\pi_3^*\alpha \cap P[i])))$

 \hspace{15mm} $P_{\beta}[-i]: H_{*}(X^{[n-i]}) \to H_{*}(X^{[n]})$

 \hspace{42mm}$\gamma \mapsto p_{1*}((p_2^*\gamma)\cap( \pi_{1,2*}(\pi_3^*\beta \cap P[-i])))$

\noindent respectively, where $\pi_j$ and $p_j$ are projections of $X^{[n-i]} \times X^{[n]} \times X$ and $X^{[n-i]} \times X^{[n]}$ (respectively) to their $j^{th}$ product factor. In other words, these operators are given by the correspondences defined by $\pi_{1,2*}(\pi_3^*\alpha \cap P[i])$ and $\pi_{1,2*}(\pi_3^*\beta \cap P[-i])$.

\begin{theorema} \cite{Nak},\cite{Gro}

(i) We have the following relations:\vspace{3mm}

\hspace{7mm}$[P_{\alpha}[i],P_{\beta}[j]]= (-1)^{i-1} i \delta_{i+j}<\alpha,\beta> Id$ \hspace{5mm} if $(-1)^{deg(\alpha) deg(\beta)}=1$

\hspace{5mm}$\{P_{\alpha}[i],P_{\beta}[j]\}=(-1)^{i-1} i \delta_{i+j}<\alpha,\beta> Id$ \hspace{5mm} otherwise

 (ii) $\displaystyle\bigoplus_{i=0}^{\infty}H_*(X^{[n]})$ is an irreducible representation of the Heisenberg superalgebra associated to $X$, with the highest weight vector being the generator of $H_0(X^{[0]})=\mathbb{Q}$
\end{theorema}

We give an explicit set of generators for the Hilbert scheme of points on the projective plane relative to a line, which will be the analogue of Nakajima's basis for the relative case. In the relative setting we will show that there are some new relations among these generators. 
\begin{theorem}\label{relationtheorem}
The cohomology groups of ${\PP^2}_{\PO}^{[n]}$ are generated by the product cycles. The following four types of relations (which are discussed in section 6) give a complete set of relations:
\begin{itemize}
\item Point-Bubble relations
\item Point-Point relations
\item Point-Line relations
\item Line-Line relations
\end{itemize}

\end{theorem}

We use the following theorem which connects the Chow groups and  the cohomology groups of the relative Hilbert scheme of points.
 
\begin{theorem}\label{chow-coh-sur}
The natural map from the Chow group to the Borel-Moore homology of ${\PP^2}_{\PO}^{[n]}$ is an isomorphism.
\end{theorem}

\subsection{Outline}
 
 In section 2 we recall the definition of moduli stack of stable ideal sheaves. In the case of ideals with zero dimensional support this moduli stack is the relative Hilbert scheme of points. 
 
 In section 3 we compute the Betti numbers of the Hilbert scheme of points for a surface relative to a smooth divisor on it. 
 
We then consider the natural torus action on the projective plane. In section 4 we give a combinatorial description of the fixed point loci. Using this description we prove that the natural map from the Chow groups to the cohomology groups of these spaces are surjective (Theorem 3). The main tool is the machinery of higher Chow groups, which we  will discuss in section 5. In section 6 we will give a complete geometric description of the relations for the Hilbert scheme of points on the projective plane relative to a line. In section 7 we show that all the relations arise in this way (Theorem \ref{relationtheorem}).

\section{Background Materials}\label{background}

\subsection{Informal description}\label{informal}

Let $S$ be a quasi-projective non-singular surface, and $D$ be a smooth Cartier divisor. We consider the Hilbert scheme of points in $S$ relative to $D$. This space is constructed in \cite{LiWu}. Since some of the ideas will be used later in our paper we will briefly describe the construction of this space. A more technical overview of the subject is given in the next subsection.

Consider the Hilbert scheme of points of $S \setminus D$, which is not a proper scheme. We construct a compactification relative to $D$. Consider the expanded degenerations of $S$ relative to $D$. More precisely take $B$ to be the $\PO$-bundle over $D$ corresponding to $\mathcal{N}_{D/S} \bigoplus \mathcal{O_{D}}$, i.e. $B=\mathbb{P}_{D}(\mathcal{N}_{D/S} \bigoplus \mathcal{O_{D}})$. Since $B$ is the projecivization of a direct sum it comes equipped with two natural sections which we call the zero section and the infinity section. Take $N$ copies of $B$ and glue the zero section of the $(i+1)^{th}$ copy to the infinity section of the $i^{th}$ copy, for $1\leq i\leq N-1$. Also glue the zero section of the first bubble to the divisor $D$ in $S$. We denote the resulting scheme by $S/D[N]$ and call it the length $N$ expanded degeneration of $S$ relative to $D$. Note that the normal bundle to the zero section in each copy of $B$ is $\mathcal{N}^{\vee}_{D/S}$, and the normal bundle to the infinity section is $\mathcal{N}_{D/S}$.

From this picture it is not hard to see that for any $N$ there is an action of $(\CS)^{N}$ on the space $S/D[N]$. The $i^{th}$ copy of $\CS$ acts on the $i^{th}$ copy of $B$ by fiber-wise multiplication. This action can be lifted to an action on the Hilbert scheme of points of $S/D[N]$ (with finite stabilizers). 

We give a moduli description for points that we add in order to compactify $(S \setminus D)^{[n]}$. As a point in $S \setminus D$ moves toward $D$, we obtain a family of subschemes of $S \setminus D$ over $\AO \setminus 0$. We call the total space of this family $F$. The support of a given point on the fiber above $t$ is moving toward $D$ as $t$ goes to zero. 

Take $B$ to be the blow up of $S\times \AO$ along $D\times 0$. Note that the fiber of $B$ over zero is the length $1$ expanded degeneration of $S$ relative to $D$. If we take the closure of $F$ in $B$ then we arrive at a flat (over $\mathbb{A}^1$) subscheme of $B$ and the fiber above zero is the candidate for the limit of the family $F$. 

There is a point that one should consider. The embedding of $F$ in $B$ is not canonical. Take any automorphism of $\AO \setminus 0$ i.e. multiplication by an element of $\CS$. This automorphism  induces an automorphism of $F$. We fix an embedding $\AO \setminus 0 \subset \AO$ (for example the canonical embedding). Hence any automorphism gives a possibly different subscheme of $B$. Note that all these families over $\AO \setminus 0$ are isomorphic, so their limit as a point in the moduli space should be the same. The limiting points of different embeddings of $F$ in $B$ differ by the natural action of $\CS$ on the bubble. Hence, in order to get a well-defined limit we have to identify the points in the orbit of the $\CS$ action. 

Therefor we should take the quotient of the Hilbert scheme of points in the bubble by the $\CS$ action and glue it to the Hilbert scheme of points on $S\setminus D$. If a point in the bubble moves toward the infinity section we may repeat the above procedure and add another bubble. 

Note that if all the points in a bubble go to the next bubble then that bubble would be empty. In this case the stability condition (i.e. having finite order of automorphism) forces that we have to delete the empty bubble. Fix $n$ to be the length of the subschemes. Hence any point in the relative Hilbert scheme of $n$ points on $S$ relative to $D$ can have at most $n$ bubbles glued to $(S \setminus D)$.

\subsection{The Stack of Expanded Degenerations}

Let $C$ be a nonsingular affine curve with a distinguished point $0\in C$. Let $\pi : X \to C$ be a flat projective family of schemes of relative dimention $d > 0$, which is smooth away from the fiber over $0$, and $X_0$ is a union of two smooth schemes intersecting transversally along a smooth divisor $D$. 

There is a $(\CS)^n$ action on $\mathbb{A}^{n+1}$ given by:

\begin{center}
$(t_1,\cdots,t_n).(a_1,\cdots,a_{n+1})=(t_1 a_1 , t_1^{-1} a_2 t_2 , \cdots,t_{n-1} a_n t_n , t_n^{-1} a_{n+1})$.
\end{center}

\noindent where $(t_1,\cdots,t_n) \in (\CS)^n$ and $(a_1,\cdots,a_{n+1}) \in \mathbb{A}^{n+1}$. Let $p:\mathbb{A}^{n+1}\to\mathbb{A}^{1}$ be the product morphism:

\begin{center}
$p(a_1,\cdots,a_{n+1})=a_1 a_2 \cdots a_{n+1}$.
\end{center}

By replacing $C$ with an open neighborhood of $0_C \in C$ we may assume that there is an \'{e}tale morphism $C \to \mathbb{A}^1$ so that $0_C \in C$ is mapped to $0_{\mathbb{A}^1} \in \mathbb{A}^1$ and $0_C$ is the only point that lies over $0_{\mathbb{A}^1}$. We fix such a map $C \to \mathbb{A}^1$ once and for all. Let $C[n] = C \times_{\mathbb{A}^1}\mathbb{A}^{n+1}$. For the family $\pi : X \to C$, $X[n]$ is defined as a desingularization of  $X \times_{\mathbb{A}^1} \mathbb{A}^{n+1}$, and is constructed in \cite[Section 1.1]{Li}. There is an induced $(\CS)^n$ action on $X[n]$ which comes from the action of $(\CS)^n$ on $C[n]$. The fiber $X[n]_0$ of this family over $0_{C[n]}$ is a semistable model of $X_0$ with $n+2$ components.




\begin{definition}

Let $S$ be a  $C$-scheme. An \emph{effective degeneration over} $S$ is a $C$-morphism from $S$ to $C[n]$. A pair $(\mathcal{X},p)$ consisting of a family $\mathcal{X}$ of schemes over $S$, and a surjective $S$-morphism $p:\mathcal{X}\to X\times_{C} S$ is called an \emph{expanded degeneration over} $S$, if there is an open covering $\lbrace S_{\alpha}\rbrace$ of $S$, such that over each $S_{\alpha}$ the restriction of $\mathcal{X}$ is isomorphic to an effective degeneration. Let $(\mathcal{X},p)$, $(\mathcal{X}',p')$ be two expanded degenerations over $S$, $S'$ respectively. An arrow $\mathcal{X} \to \mathcal{X}'$ consists of a $C$-morphism $S \to S'$ and an $S$-isomorphism $\mathcal{X} \to \mathcal{X}' \times_{S'} S$  compatible with their projections to $ X \times_{C} S$. Let $\mathfrak{X}$ be the category whose objects are expanded degenerations $(\mathcal{X},p)$ and its morphisms are such arrows.
\end{definition}

There is a functor $\mathcal{F} : \mathfrak{X} \to Sch/C$ that sends an expanded degeneration $(\mathcal{X},p)$ to the base scheme $S$ of the family $\mathfrak{X}$. Then $\mathfrak{X}$ together with the functor $\mathcal{F}$ is a groupoid over $C$.

\begin{proposition}\cite[Prop. 1.10]{Li} The groupoid $\mathfrak{X}$ is a stack over $C$.
\end{proposition}

\subsection{Admissible Ideal Sheaves}

\begin{definition}

Let $W$ be a scheme and $D \subset W$ a divisor. An ideal sheaf $\mathcal{I} \subset \mathcal{O}_W$ is called \emph{normal} to $D$ if the canonical homomorphism  $\mathcal{I}\otimes\mathcal{O}_D\to \mathcal{O}_D$ is injective.
\end{definition}
%
%

Normality is a local property. We say that $\mathcal{I}$ is normal to $D$ at a closed point $p\in D$ if the canonical homomorphism of stalks  $\mathcal{I}_p\otimes\mathcal{O}_{D,p}\to \mathcal{O}_{D,p}$ is injective. Then $\mathcal{I}$ is normal to $D$ if and only if it is normal at every closed point $p \in D$.

Since completion is an exact functor on the category of finitely generated modules over a Noetherian ring, we have:

\begin{lemma}\label{completion}

Let $\mathcal{I}$ be an ideal sheaf on $Y$. $\mathcal{I}$ is normal at $p\in D$ if and only if the canonical homomorphism $\hat{\mathcal{I}}_p\otimes\hat{\mathcal{O}}_{D,p}\to \hat{\mathcal{O}}_{D,p}$ is injective.
\end{lemma}


\begin{corollary}\label{relative remark}
Let $Y$ be a smooth variety, $D$ a smooth divisor in $Y$ and $\mathcal{I}_Z\subset \mathcal{O}_Y$ the ideal sheaf of a closed zero dimensional subscheme. $\mathcal{I}$ is normal to $D$ if and only if $Z\cap D = \emptyset$.
\end{corollary}
\begin{proof}
Let $\mathcal{I}$ be normal to $D$ and $p\in Z\cap D$. Since $Z$ is zero dimensional, we can pick an analytic open neighborhood $U$ of $p$ such that $Z\cap U =p$. Take $f$ to be a regular function on $U$ such that $D=(f=0)$. Then $f\otimes 1$ is a nonzero element of  $\hat{\mathcal{I}}_p\otimes\hat{\mathcal{O}}_{D,p}$ that is sent to zero by $\hat{\mathcal{I}}_p\otimes\hat{\mathcal{O}}_{D,p}\to \hat{\mathcal{O}}_{D,p}$, which contradicts the result of lemma \ref{completion}. Hence $Z\cap D = \emptyset$. 

For the converse, we assume that $Z\cap D = \emptyset$. Then for any $p\in D$ there is an open analytic neighborhood that does not intersect $Z$. Hence we have  $\hat{\mathcal{I}}_p = \hat{\mathcal{O}}_p$ and the converse follows.
\end{proof}

\begin{lemma}\cite[Prop. 3.7]{LiWu}\label{normal-degeneration}
Let $W=W_1\cup W_2$ be the union of two smooth subschemes that intersect transversely along $D$. $\mathcal{I}$ is normal to $D$ if and only if for $i=1,2$, $\mathcal{I}|_{W_i}$ is normal to $D$.
\end{lemma}

Let $X[n]_0=B_0\cup B_1\cup \cdots \cup B_{n+1} $ be a semistable model of $X_0$. The singular locus of $X[n]_0$ is the disjoint union of $n+1$ copies of $D$ which we denote by $D_i$ for $i=0,\cdots,n$. 

\begin{definition}
An ideal sheaf $\mathcal{I}$ on $X[n]_0$ is \emph{admissible} if it is normal to $D_i$ for $i=0,\cdots,n$.
\end{definition}

\begin{lemma}
Let $\mathcal{I}$ be an ideal sheaf on $X[n]_0$. 
\begin{enumerate}
\item If $\mathcal{I}$ is admissible then $\mathcal{I}\otimes\mathcal{O}_{B_k}$ is an ideal sheaf on $B_k$, and it is normal to the special divisors $D_{k-1}$ and $D_k$ of $B_k$ for all $k$.
\item Conversely, for every $k$ let $\mathcal{I}_k$ be an ideal sheaf on $B_k$ which is normal to $D_{k-1}$, $D_k$ and $\mathcal{I}_k\otimes\mathcal{O}_{D_k}=\mathcal{I}_{k+1}\otimes\mathcal{O}_{D_k}$. Then there is an admissible ideal sheaf $\mathcal{I}$ on $X[n]_0$ such that $\mathcal{I}\otimes\mathcal{O}_{D_k}=\mathcal{I}_{k}$.
\end{enumerate}
\end{lemma} 

\begin{proof}
The result follows from lemma \ref{normal-degeneration} by induction on $n$.
\end{proof}

Since we study the case where the dimension of the support is zero, by corollary ~\ref{relative remark} we know that the support of these ideals have no intersection with the special divisors in the bubbles. So the data of an ideal sheaf normal to all the $D_i$'s over a semistable model $X[n]_0$ would consist of ideal sheaves on each bubble and the base, such that their supports are disjoint from the special divisors. 

\subsection{Stack of Stable Ideal Sheaves}

\begin{definition}\cite[Definition 4.2]{LiWu}
 An \emph{automorphism} of an ideal sheaf $\mathcal{I}$ over $X[n]_0$ is an isomorphism of  $X[n]_0$ that fixes $\mathcal{I}$. $\mathcal{I}$ is called \emph{stable} if it is admissible and has finitely many automorphisms.

\end{definition}

\begin{definition}\cite[Definition 3.9,4.2]{LiWu}
Let $\mathcal{X}/S$ be an expanded degeneration, and $P$ be a fixed polynomial. Let $\phi: \mathcal{I} \to \mathcal{O}_{\mathcal{X}}$ be an ideal sheaf on $\mathcal{X}$ such that $coker{\phi}$ is $S$ flat. $\mathcal{I}$ is called \emph{admissible} if $\mathcal{I}_s=\mathcal{I}|_{\mathcal{X}_s}$ is admissible for every closed point $s\in S$. It is called stable if $\mathcal{I}_s$ is stable for every closed point $s\in S$. It is called a family of stable ideal sheaves of type $P$ on $\mathcal{X}/S$ if it is stable and the  Hilbert polynomial of each fiber is $P$.

\end{definition}

We are now ready to define the stack of stable ideal sheaves of type $P$. The objects of $\mathcal{X}^{P}_{X/C}$ are the pairs  $(\mathcal{X}/S,\mathcal{I})$ of a family of stable ideal sheaves of type $P$ on $\mathcal{X}/S$. A map between two objects $(\mathcal{X}_1/S_1,\mathcal{I}_1)$ and $(\mathcal{X}_2/S_2,\mathcal{I}_2)$, is given by a map between expanded degeneration $(\mathcal{X}_1/S_1)$ and $(\mathcal{X}_2/S_2)$ such that the induced map on the ideal sheaves is an isomorphism. 

\begin{theorem}\cite[Prop. 4.14]{LiWu} $\mathcal{X}^{P}_{X/C}$ is a proper Deligne-Mumford stack of finite type over $C$.
\end{theorem}

\section{Betti Numbers of the Relative Hilbert Scheme of Points}\label{vir-betti}

In order to compute the analogue of G\"{o}ttsche's formula for the relative case we need the following theorem (see for example \cite{Durfee}):

\begin{theorem}\label{virtual-p}

To any complex algebraic variety $X$ one can assign a virtual
Poincar\'{e} polynomial $P_{X}(t)$ with the following properties:

\begin{enumerate}
\item $P_{X}(t)= \sum rank(\mathrm{H}^{i}(X)) t^{i}$ if $X$ is non-singular and projective.
\item $P_{X}(t)=P_{Y}(t)$+$P_{U}(t)$ if $Y$ is a closed algebraic subset of $X$ and $U=X\setminus Y$.
\item If $X$ is a disjoint union of a finite number of locally closed subvarieties $X_{i}$, then $P_{X}(t)=\sum
P_{X_{i}}(t)$.
\item If $X \rightarrow Y$ is a bundle with fiber $F$, which is locally trivial in the Zarisky topology, then $P_{X}(t)= P_{Y}(t) \cdot P_{F}(t)$.
\end{enumerate}

\end{theorem}

\begin{example}

For $\PO$ the virtual Poincar\'{e} polynomial and the Poincar\'{e} polynomial coincide. Hence we have $P_{\PO}(t)=t^2+1$, and $\CS$ can be
obtained by removing two points from $\PO$. Thus by the second property we have  $P_{\CS}(t)=t^2-1$.

If $D$ is a curve with Betti numbers $a_0$,$a_1$ and $a_2$, then the Betti numbers of $D\times \CS$ can be computed to give: $P_{D \times\CS}(t)=(a_2t^2+a_1t+a_0)(t^2-1)$.
\label{ex:no1}
\end{example}

\begin{remark}
The virtual Poincar\'{e} polynomial can be defined for singular varieties as well. Deligne in \cite{De} and Gillet and Soul\'{e} in \cite{GSo} show that for any complex algebraic variety one can define virtual Betti numbers that satisfy property 1-4 of theorem \ref{virtual-p}. They show that $$P_{X}=\displaystyle \sum_{i,j} (-1)^{i+j}  dim ( gr_{W}^{j}H^i_{c}(X))t^{j} $$ satisfies these properties. In \cite[Section 4]{Joyce} Joyce proved that this machinery can be extended to work in the category of Artin stacks.
\end{remark}

\begin{theorem}
Let $S$ be a smooth quasi-projective surface, and $D$ be a smooth Cartier divisor. The generating function for the normalized Poincar\'{e} polynomial of the relative Hilbert scheme of points is given by:

$$\sum q^{n}\hat{P}_{S^{[n]}_{D}}(t) = \frac{(t^{2}-1)\hat{H}_{S}(q,t)}{t^{2}C_{D}(q,t)-C_{D}(q,t^{-1})},$$

\noindent where $\hat{H}_{S}(q,t)$ is the normalized Poincar\'{e} polynomial of the Hilbert scheme of points on $S$ and $$C_{D}(q,t)=\displaystyle\prod_{m=1}^{\infty}\frac{(1+t^{-1}q^{m})^{b_{1}(D)}}{(1-t^{-2} q^{m})^{b_{0}(D)}(1-q^{m})^{b_{2}(D)}} .$$

\end{theorem}

\begin{proof}

A point $p$ of the relative Hilbert scheme corresponds to a subscheme of an expanded degeneration. Such a subscheme is the disjoint union of:
\begin{itemize}
\item Point with support in $S\setminus D$
\item Point with support in the $i^{th}$ bubble for some $1\leq i \leq n$ (which is isomorphic to the total space of the
normal bundle of $D$ in $S$ minus the zero section)
\end{itemize}

We call these the components of $p$.

The stability condition translates as follows. If the $j^{th}$ bubble of the corresponding expanded degeneration is empty (i.e. has no point supported on it), then all the bubbles with index greater than $j$ are also empty. 

Therefore for such $p$ the components of $p$ can be considered as:
\begin{itemize}
\item A point in the Hilbert scheme of points of $S\setminus D$.
\item A point in the Hilbert scheme of points of the $i^{th}$ copy of the total space of the
normal bundle of $D$ in $S$ minus the zero section (for some $1\leq i \leq n$).
\end{itemize}

The discussion in the beginning of section 2 imposes a condition. The point in the Hilbert scheme of points on the $i^{th}$ copy of the total space of the normal bundle is defined up to the $\CS$ action i.e. if two such point can be obtained from each other by acting $\CS$, then as components of $p$ they are the same. 

Hence we obtain a stratification of the relative Hilbert scheme of points. Each strata is the product of the Hilbert scheme of points on $S \setminus D$ and a number of copies of the quotient of the Hilbert scheme of points of the total space of the
normal bundle of $D$ in $S$ minus the zero section by the action of $\CS$.


  For a given surface $Y$ we use $H_{Y}(q,t)$ to denote $\sum q^{n}P_{Y^{[n]}}(t)$. The virtual Poincar\'{e} polynomial of the Hilbert scheme of the  projectivized normal bundle without zero and infinity section is $H_{\mathcal{N}_{D/S}^{o}}(q,t)$. Since we consider the non-empty bubbles we get $H_{\mathcal{N}_{D/S}^{o}}(q,t)-1$. Each bubble is obtained by taking quotient with the $\CS$. By Theorem 5.4 from \cite{GP} the virtual Poincar\'{e} polynomial of the quotient space is the quotient of $H_{\mathcal{N}_{D/S}^{o}}(q,t)-1$ by $t^2-1$ (the virtual Poincar\'{e} polynomial of $\CS$). Hence the virtual Poincar\'{e} polynomial of a given bubble is $$\displaystyle\frac{H_{\mathcal{N}_{D/S}^{o}}(q,t)-1}{t^{2}-1}.$$ We have a stratification of the relative Hilbert scheme according to the number of bubbles. The virtual Poincar\'{e} polynomial of the part with $i$ bubbles is $$H_{S\setminus D}(q,t) \displaystyle\left(\frac{H_{\mathcal{N}_{D/S}^{o}}(q,t)-1}{t^{2}-1}\right)^{i} .$$ By the above discussion the relative Hilbert scheme is stratified by such parts, therefor:\vspace{3mm}

$$\sum q^{n}P_{S^{[n]}_{D}}(t)=H_{S\setminus D}(q,t) \displaystyle\sum_{i=0}^{\infty}\left(\frac{H_{\mathcal{N}_{D/S}^{o}}(q,t)-1}{t^{2}-1}\right)^{i}.$$

Note that for the Hilbert scheme of $n$ points (the coefficient of $q^n$) there is no contribution from terms with exponent larger than $n$ in the sum, which reflects the fact that we can not have more than $n$ bubbles for a subscheme of length $n$. 

If we define $C_{D}(q,t):=\displaystyle\prod_{m=1}^{\infty}\frac{(1+t^{-1}q^{m})^{b_{1}(D)}}{(1-t^{-2} q^{m})^{b_{0}(D)}(1-q^{m})^{b_{2}(D)}}$,
then by G\"{o}ttsche's formula and example ~(\ref{ex:no1}) we find:\vspace{5mm}

$\left\{
\begin{array}{lcl}
 H_{S\setminus D}(q,t)&=&\displaystyle\prod_{m=1}^{\infty}\frac{(1+t^{2m-1}q^{m})^{b_{1}(S)-b_{1}(D)}(1+t^{2m+1}q^{m})^{b_{3}(S)}}{(1-t^{2m-2}
 q^{m})^{b_{0}(S)-b_0(D)}(1-t^{2m}q^{m})^{b_{2}(S)-b_2(D)}(1-t^{2m+2}q^{m})^{b_{4}(S)}}\\
 &&\\
&=&\displaystyle\frac{H_{S}(q,t)}{C_{D}(qt^2,t)}\\
&&\\
H_{D\times\CS}(q,t)&=&\displaystyle\prod_{m=1}^{\infty}\frac{(1+t^{2m-1}q^{m})^{-b_{1}(D)}(1+t^{2m+1}q^{m})^{b_{1}(D)}}{(1-t^{2m-2}q^{m})^{-b_{0}(D)}
(1-t^{2m}q^{m})^{b_0(D)-b_{2}(D)}(1-t^{2m+2}q^{m})^{b_{2}(D)}}\\
&&\\
&=&\displaystyle \frac{1}{C_{D}(qt^2,t)}C_{D}(qt^2,t^{-1})
\end{array}
\right.$
\vspace{5mm}

Note that the change of variable from $q$ to $qt^2$ corresponds to writing the Poincar\'{e} polynomial in the normalized form. Hence $H_{Y}(q,t)=\hat{H}_{Y}(qt^2,t)$ and we can summarize all this computation as follows:



\begin{eqnarray}
\sum q^{n}\hat{P}_{S^{[n]}_{D}}(t)&=&\hat{H}_{S\setminus D}(q,t) \left(\displaystyle\sum_{i=0}^{\infty}\left(\frac{\hat{H}_{\mathcal{N}_{D/S}^{o}}(q,t)-1}{t^{2}-1}\right)^{i}\right) \nonumber\\
&=&\frac{\hat{H}_{S}(q,t)}{C_{D}(q,t)}\left(\sum_{i=0}^{\infty}\left(\frac{\hat{H}_{D\times \CS}(q,t)-1}{t^{2}-1}\right)^{i}\right)\nonumber\\
&=&\frac{\hat{H}_{S}(q,t)}{C_{D}(q,t)}\left(\frac{1}{1-\frac{\frac{C_{D}(q,t^{-1})}{C_{D}(q,t)}-1}{t^{2}-1}}\right)\nonumber\\
&=& \frac{(t^{2}-1)\hat{H}_{S}(q,t)}{t^{2}C_{D}(q,t)-C_{D}(q,t^{-1})}. \label{eq:no1}
\end{eqnarray}

\end{proof}

\begin{example}\label{plane example}

For $S=\PP^2$ and $D$ a line by Eq.~(\ref{eq:no1}) we have:
$$
\begin{array}{lll}
\hat{H}_{{\PP^2}^{[n]}_{\PO}}(q,t)&=&\frac{\displaystyle(t^{2}-1)\prod_{m=1}^{\infty}\frac{1}
{(1-t^{-2} q^{m})(1-q^{m})(1-t^{2}q^{m})}}{\displaystyle t^{2}\prod_{m=1}^{\infty}\frac{1}
{(1-t^{-2} q^{m})(1-q^{m})}-\displaystyle\prod_{m=1}^{\infty}\frac{1}
{(1-q^{m})(1-t^{2}q^{m})}} \\
&&\\
 &=&\displaystyle\frac{t^{2}-1}{\ t^{2}\prod_{m=1}^{\infty}(1-t^{2}q^{m})-\prod_{m=1}^{\infty}(1-t^{-2}q^{m})}.
\end{array}
$$

\end{example}

\section{Torus Action on the Relative Hilbert Scheme of Points and The Fixed Point Loci}

From now on we consider the case of projective plane and a line on it. We start by taking the natural $(\CS)^3$ action on the relative Hilbert scheme of points, and give a description of the fixed point loci. If the relative divisor is given by $\{x_0=0\}$, we consider the following action on the projective plane:

 \begin{center}$(t_0,t_1,t_2).[x_0;x_1;x_2]\mapsto[t_0 x_0;t_1 x_1;t_2 x_2]$.\end{center}

 This action fixes the relative divisor, and induces an action on the normal bundle of this divisor. Hence it induces an action on each bubble, which gives us an induced action on the whole relative Hilbert scheme.

 Pick a fixed point $p\in {\PP^2}^{[n]}_{\PO}$. Then $p$ has a part supported on the projective plane, and a part that is supported on the bubbles. The part with support on the plane is supported on $[1;0;0]$ since this is the only fixed point on the plane. Since $p$ is fixed under the action, its support is also fixed. Locally this part is a subscheme of $\C^2$ that is fixed under the natural $(\CS)^2$ action, i.e. a homogenous ideal supported at the origin. If we fix $k$ to be the length of this part, any such homogenous ideal can be parameterized by a Young tableau of length $k$. 
 
 With the same argument the part supported on the bubbles can have its support (on each bubble) only on the fibers above the zero and infinity. So if we look at one of the bubbles and restrict our attention to the fiber above zero then the local picture is $\C^2$ with only one $\CS$ action on one of the coordinates, i.e. $t.(x,y)\mapsto(tx,y)$. If $I$ is a fixed ideal we can pick a set of generators which is fixed under the action (up to scalar). This means that each generator is homogenous with respect to $x$.

 If $\{x^i f_i(y)\}$ is such a set then $f_i | f_j$ for each $j<i$. This means if we fix a root of $f_0$, and denote by $a_i$ the multiplicity of this root in $f_i$, then we have $...<a_2<a_1<a_0$. To each root of $f_0$ we can associate a Young tableau with $a_0$ boxes in the first column and $a_1$ boxes in the second column and so on. For example if we consider $\mathcal{I}=\langle xy^5(y-1)^4,x^2y^2(y-1),x^3y^1\rangle$ the corresponding diagrams for the roots 0 and 1 are:
  
\begin{figure}[h]
\centerline{\includegraphics[scale=.45]{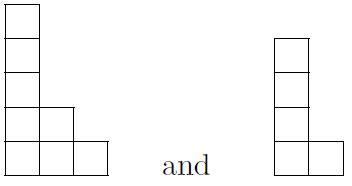}}
\end{figure}
%
%
%
%
 
 The Young tableau associated with a deformation of this ideal, where two of these roots come together will be the sum of the Young tableaux associated with these roots.
 
\begin{figure}[h]
\centerline{\includegraphics[scale=.45]{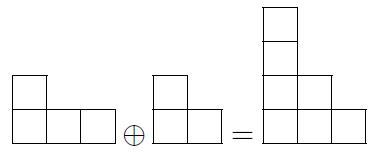}}
\end{figure}

 For a given fixed point in each bubble we get two such ideals, one above zero and one above infinity. Fix the combinatorial data of the Young tableaux of all the roots of both of these ideals. If they have $k$ and $l$ distinct roots (respectively), then the closure of the locus of such ideals in the relative Hilbert scheme is isomorphic to the quotient of the  moduli space of $k$ unordered red and $l$ unordered blue points in $\CS$ by the $\CS$ action (The $\CS$ acts by dilation). Note that these points can come together. The moduli space of $k$ unordered points in $\CS$ is $\CS \times (\C)^{k-1}$ since these points can be thought of as the roots of a monic polynomial of degree $k$ with non zero constant term. So the closure of this locus is isomorphic to $(\CS)^2 \times (\C)^{k+l-2}$ modulo $\CS$ (if $k$ and $l$ are both nonzero), and if one of them, say $l$, is zero it is isomorphic to $(\CS) \times (\C)^{k-1}$ modulo $\CS$. Using the $\CS$ action we can set the constant term of one of these polynomials equal to 1, say the one with $k$ roots, and since $\CS$ acts on the constant term by $ t . b_0 \mapsto t^k b_0$ the ambiguity is a $k^{th}$ root of unity.

 In sum, for any fixed point we can consider the combinatorial data associated to it, and this gives us a stratification of the fixed point loci into parts which are isomorphic to the product of quotients of $(\CS)^a \times (\C)^{b}$ by a finite group action. 
 
 \begin{remark}
 Note that the above description of the fixed locus shows that it is a smooth DM-stack which allows us to use the localization theorem in the next section.
 \end{remark}

\section{Chow-Cohomology Correspondence}

The main theorem of this section is the following:

\begin{theorem}\label{chow-coh-surjective}

The natural map between the Chow group and the Borel-Moore homology of ${\PP^2}_{\PO}^{[n]}$ is an isomorphism.
\end{theorem}

The strategy of the proof is to relate the cohomology (Chow) groups of the relative Hilbert scheme to the cohomology (Chow) groups of the fixed point locus. This is done by the localization formula of Atiyah and Bott. The localization formula for cohomology ring of compact spaces is proven in \cite{AB,Bo}. The case of Chow groups is done in \cite{EG}. The case of Deligne-Mumford stacks is covered in \cite{GP2}. 

Let $X$ be a space with a $G=(\CS)^n$ action. We have $H_G^{*}(pt)=Q[t_1,\cdots,t_n]$. Let $U$ be set of homogeneous elements in the ideal $<t_1,\cdots,t_n> \subset Q[t_1,\cdots,t_n]$. The localization theorem for the equivariant cohomology says that the equivariant inclusion map $i_*:H^{*}(X^{G})\otimes U^{-1} \to H_{G}^{*}(X)\otimes U^{-1}$ is an isomorphism. 
 
  We show that for the fixed point locus the natural map between Chow and cohomology is an isomorphism. This theorem with the localization theorem show that the natural map between the equivariant Chow group and the equivariant Borel-Moore homology of the relative Hilbert scheme of points of the projective plane and a line on it is an isomorphism. Hence by taking the non-equivariant elements of both sides we otain the proof of Theorem ~\ref{chow-coh-surjective}. 

So far we reduced the proof of Theorem ~\ref{chow-coh-surjective} to the proof of the following theorem. 

\begin{theorem}
The natural map between the Chow group and the Borel-Moore homology of the $\CS$ fixed point locus of ${\PP^2}_{\PO}^{[n]}$ is an isomorphism.
\end{theorem}

In order to prove this theorem we prove that this map is an isomorphism for a class of stacks containing the fixed point locus as an element. More precisely :

\begin{definition}
The class of linear stacks is the smallest class of stacks that contains quotient stacks of affine spaces of any dimension by the action of a finite group, with the property that:
\begin{itemize}
\item The complement of any linear stack embedded in the quotient stack of affine space (by a finite group) is a linear stack.
\item Any space that can be stratified as a finite disjoint union of linear stacks is a linear stack.
\end{itemize}

\end{definition}

For the definition and basic properties of the morphism between Chow group and the Borel-Moore homology for DM-stacks see \cite{Olsson}. 
 
\begin{theorem}\label{chow-coh}

For any linear stack $X$ over the complex numbers, the natural map:

\begin{center}$CH_{i}X\otimes\mathbb{Q} \to W_{-2i}H_{2i}^{BM}(X,\mathbb{Q})$\end{center}

from the Chow groups into he smallest space of Borel-Moore homology with respect to the weight filtration, is an isomorphism. 
\end{theorem}

We will follow the argument of Totaro in \cite{Totaro} in which he proved the same theorem holds for linear varieties. The argument uses the machinary of higher Chow groups. For definitions and basic properties of higher Chow groups of schemes see \cite{Bloch}. For example the $CH(X,0)$ is the Chow groups of $X$ as defined by Fulton \cite{Fulton}. In \cite{Joshua} Joshua proved that the similar machinery works for the stacks. 

\begin{definition}
A stack $X$ satisfies:

\begin{itemize}
\item the \emph{weak property} if $CH^{dimX-i}(X,0)\otimes\mathbb{Q} \to W_{-2i}H_{2i}^{BM}(X,\mathbb{Q})$ is an isomorphism.

\item the \textit{strong property} if it satisfies the weak property and also the map

$CH^{dimX-i}(X,1)\otimes\mathbb{Q} \to gr^{W}_{-2i}H_{2i}^{BM}(X,\mathbb{Q})$ is surjective.
\end{itemize}

\end{definition}

Since we only work with Chow groups with coefficients in $\mathbb{Q}$, by abuse of notation we denote the Chow groups of $X$ with coefficients in $\mathbb{Q}$ by $CH^{i}(X,j)$.

\begin{lemma}
Let $X$ be a given stack and $S$ be substack of $X$ that satisfies the weak property, and let $U=X-S$.

a) If $X$ satisfies the strong property then $U$ also satisfies the strong property.

b) If $U$ satisfies the strong property then $X$ satisfies the weak property.
\end{lemma}
\begin{proof}

a) We have the following exact sequences:

\begin{small}
\begin{flushleft}
$\begin{array}{ccccccccccccc}
\\
CH^{dimX-i}(X,1)& \rightarrow & CH^{dimU-i}(U,1)& \rightarrow & CH^{dimS-i}(S,0)& \rightarrow & CH^{dimX-i}(X,0)& \rightarrow & CH^{dimU-i}(U,0)& \rightarrow &0\\
\downarrow &  & \downarrow &  & \downarrow & & \downarrow &  & \downarrow  &  & \downarrow\\ 
gr_{-2i}^{W}H^{BM}_{2i+1}(X,\mathbb{Q}) & \rightarrow & gr_{-2i}^{W}H^{BM}_{2i+1}(U,\mathbb{Q}) & \rightarrow & W_{-2i}H^{BM}_{2i}(S,\mathbb{Q}) & \rightarrow & W_{-2i}H^{BM}_{2i}(X,\mathbb{Q}) & \rightarrow & W_{-2i}H^{BM}_{2i}(U,\mathbb{Q})& \rightarrow &0\\
\\
\end{array}$
\end{flushleft}
\end{small}

In this diagram the first column is surjective, and the third and forth column are isomorphisms, so by diagram chasing we see that the second column is surjective and the fifth column is an isomorphism. Hence $U$ satisfies the strong property.

b) By the same argument in this case by assumption the second column is surjective, and the third and fifth are isomorphism so the fourth column is surjective. Thus $X$ satisfies the weak property.
\end{proof}

\begin{proof}[\textbf{Proof of Theorem ~\ref{chow-coh}}]

Consider the qoutient map from the affine space to the quotient stack. By the functoriality of the morphism between Chow group and the Borel-Moore homology (see \cite{Olsson}) we obtain a commutative diagram which shows that the stack quotient of the affine space satisfies the weak and strong properties. Hence by the previous lemma we see that every linear stack satisfies the weak property. Thus the natural map $CH^{dimX-i}(X,0)\otimes\mathbb{Q} \to W_{-2i}H_{2i}^{BM}(X,\mathbb{Q})$ is an isomorphism.
\end{proof}

\section{Projective Plane and One line}

\subsection{Generators and Relations}

In this section we work with $\PP^2$ and a line as the special divisor. We extend Nakajima's notation for the cohomology classes to the relative Hilbert scheme of points. The homology group of the $\PP^2$ has three generators and we denote them by
$\alpha_i$ for $i=0,1,2$, where $\alpha_i$ is the cycle with dimension $i$. We denote the classes with support in the $k^{th}$ bubble by $\beta_i^{k}$ for $i=0,1$. More precisely, consider the locus in the Hilbert scheme of points of ${N}_{D/S}^{o}$. The Nakajima cycle associated to $\beta_i$ for $i=0,1$ gives a subscheme $B_i$ of the $Hilb({N}_{D/S}^{o})$. Take $\beta_i^{k}$ to be the closure of the image of $B_i$ in the Hilbert scheme of the $k^{th}$ bubble.

We represent a cohomology class as a product of $\alpha$'s and $\beta$'s, in order to show the support of the points in that cohomology class.
If a point with support in a given cycle has multiplicity, we show that by putting that number in the bracket. It means that $$\displaystyle \prod_{i \in A, j\in B}\alpha_{a_i}[p_i]\beta_{b_j}^{n_j}[m_j]$$ represents the Chow class such that for $i \in A$ there is a point supported on a representative of $\alpha_{a_i}$ with multiplicity $p_i$. Similarly for each $j \in B$ there is a point supported on a representative of $\beta_{b_j}$ in the $n_j^{th}$-bubble, with multiplicity $m_j$. We call these classes the \emph{product classes}.

In the construction of the relative Hilbert scheme we started by gluing certain Hilbert schemes and we took the quotient by the $\CS$ actions. For those Hilbert schemes by the Nakajima theorem we have a set of generators for the cohomology and by the construction they were all Chow classes as well. One can see that the product classes give us a set of generators for the invariant part of the Chow group, i.e. the Chow group of the quotient. By the result of Theorem ~\ref{chow-coh-sur}, we know that these cycles will also give us a set of generators for the cohomology of the relative Hilbert scheme of points in the projective plane.

\begin{example}
If $n=1$ then the relative Hilbert scheme is just the Hilbert scheme of $\PP^2$, and we have: $\beta_i^{1}={\alpha_i}$ for $i=0,1$.

If $n=2$ the situation is not as simple as in the previous case. For example we have $\alpha_0[2]=\beta_0^1[2]$. Fix a line in the $\PP^2$ and move the fat point of multiplicity 2 along
this line. We get a family over $\PO$
with $\alpha_0[2]$ and $\beta_0^1[2]$ as the fibers above $0$ and $\infty$. But there are more complicated relations among these generators. Consider $\alpha_0 \alpha_1$, the locus of points in the Hilbert scheme of 2-points, where one point is supported in a given line and the other is supported in a given point. If we move the point defining $\alpha_0$ along a line towards
the special divisor, then we will get a family over $\PO$ with fiber over zero equal to the class that we started with, and the fiber over infinity
will be $\alpha_1 \beta_0^1 +\beta_0^1 \beta_0^1$. Since when the fixed point goes to the special divisor, the point with support in the line either is still
 in the $\PP^2 \setminus D$, or is supported at the intersection point of the line and the special divisor.

\end{example}

\subsubsection{Point-Bubble Relations}

The first type of relations is obtained by moving the points with support in $\PP^2 \setminus D$ to the first bubble. Given a cohomology class with a point supported on a zero-cycle $p$ or a line $\ell\neq D$. In the first case we pick a line in $\PP^2$ that passes through $p$ and move this point along it. This family has a natural projection to $\PO$ which is given by the locus of support of this point. This gives us a family over $\PO$ with fiber over zero being the class that we started with, and as we go toward infinity the point moves toward the special divisor.

 More precisely if we start with a cycle of the form $$\alpha=\displaystyle\mathop{\prod_{i\in A_2}}_{j \in A_1, k\in A_0}\alpha_2[i]\alpha_1[j]\alpha_0[k]. \beta$$ where $\beta=\displaystyle\prod_{i\in B}\beta_{a_i}^{n_i}[m_i]$ is the part with support in the bubbles, and $A_0,A_1,A_2$ are multisets, and $B$ is an index set. We move one of the points with support at a zero-cycle, which is represented by $\alpha_0[a]$, to the special divisor. 
 
 In other words consider a cycle $C$ in the class $$\displaystyle\mathop{\prod_{k\in A_0 \setminus \{a\}}}_{j \in A_1,i\in A_2 }\alpha_2[i]\alpha_1[j]\alpha_0[k] \alpha_1[a].\beta .$$ Let $\gamma$ be the representative of $\alpha_1$ that the point with multiplicity $a$ is supported over it. By considering the location of the point with multiplicity $a$ along $\gamma$ one gets a map from $C$ to $\PO$. The fiber over zero of the family is the cycle $\alpha$ that we started with. The fiber over infinity $C_{\infty}$ is obtained as the point with multiplicity $a$ goes to $D$. By the moduli description of the points of the relative Hilbert scheme the limit of the family is given as follows:
 
 Let $T$ be the expanded degeneration of $\PP^2$ which contains $C$. We blow up $T \times \PO$ along the subscheme $D\times \{ \infty \}$. The exceptional divisor is the projectivization of the normal bundle of $D\subset \mathbb{P}^2$. Hence we get a configuration with an extra bubble. The new bubble is attached to the base $\mathbb{P}^2$. Hence the index of all the bubbles in $\beta$  is shifted by one. For the points supported on the first bubble of $C_{\infty}$ we have:
 
 \begin{itemize}
 \item The point with multiplicity $a$ is moved to the special divisor. Hence we have a point supported over the intersection of  $D$ and $\gamma$ in the first bubble.
 \item Each point in $\PP^2 \setminus D$ which was supported on a line or on the whole plane might go to the the intersection of that cycle and $D$. If this point is supported on a line, it will be supported over the intersection of that line and the special divisor. If the point is supported on the whole plane (i.e. $\alpha_2[i]$) it will be supported on the whole special divisor.
 \end{itemize}

By looking at fibers over zero and infinity, in this case we get the following relation:

\[\displaystyle\mathop{\prod_{k\in A_0 \setminus \{a\}}}_{j \in A_1,i\in A_2 }\alpha_2[i]\alpha_1[j]\alpha_0[k] \alpha_0[a].\beta=\mathop{\prod_{k\in A_0 \setminus \{a\} }}_{j \in A_1,i\in A_2 }(\alpha_2[i]+\beta_1^1[i])(\alpha_1[j]+\beta_0^1[j])\alpha_0[k] \beta_0^1[a]. \beta^{+1}\]

\noindent where $\beta^{+q}=\displaystyle\prod_{i\in B}\beta_{a_i}^{n_i+q}[m_i]$ for $q \in \mathbb{N}$.

In the second case the point is supported on a line. We denote this line by $\ell$ and the intersection of this line with the special divisor by $p$. The projectivized tangent space at $p$ is $\PO$ and so we can rotate $\ell$ around $p$ in this $\PO$.  

 In  this case we move a point supported on a one-cycle to the special divisor. In other words we consider the cycle $C'$ in the class $$\displaystyle\mathop{\prod_{k\in A_0 \setminus \{a\}}}_{j \in A_1,i\in A_2 }\alpha_2[i]\alpha_1[j]\alpha_0[k] \alpha_2[a].\beta .$$ We swipe $\mathbb{P}^2$ by rotating the line $\ell$. By projection to the location of the point with multiplicity $a$ we land on a rotated copy of $\ell$. This way we get a family over $\PO$. The fiber over zero is the cohomology class that we started with. The fiber over infinity is computed as in the previous case. It consists of cycles with these properties:

 \begin{itemize}
 \item The point that we moved to the special divisor will go to a point on the first bubble supported on the whole line.
 \item Each point with support in the bubbles, say the $i^{th}$-bubble, will go to the next bubble ($(i+1)^{th}$-bubble). 
 \item A point in $\PP^2 \setminus D$ which was supported on a line or on the whole plane might go to the first bubble, and the the new point will be supported on a point or the whole divisor (respectively).
 \end{itemize}

we have the following relation:\vspace{2mm}

\[\displaystyle\mathop{\prod_{j \in A_1}}_{k\in A_0,i\in A_2 }\alpha_2[i]\alpha_1[j]\alpha_0[k].\alpha_1[a] \beta=\displaystyle\mathop{\prod_{j \in A_1\setminus \{a\}}}_{k\in A_0,i\in A_2 }(\alpha_2[i]+\beta_1[i])(\alpha_1[j]+\beta_0[j])\alpha_0[k].\beta_1[a] \beta^{+1}\]

\subsubsection{Point-Point Relations}

If we have two points supported on zero-cycles in one of the bubbles, we can move one of them towards the infinity.

We call these two points $P_1$ and $P_2$ and also assume that they are supported in the $i^{th}$ bubble Hence we start with a cycle that can be represented as 

$$\alpha_{pp}=\displaystyle\prod_{j \in A_1, k\in A_0}\beta_0^i[k]\beta_1^i[j].\alpha\beta^{(<i)}\beta^{(>i)}$$

\noindent where $\alpha$ is the part supported on the base, and $\beta^{(<i)}$ and $\beta^{(>i)}$ are the parts supported in bubbles with index smaller (larger) that $i$ (respectively)

Using the $\CS$ action on the $i^{th}$ bubble we can fix the support of the second point. More precisely, we know that $P_2$ is supported on the fiber above a  point of the special divisor. We call it $P_3 \in D$. We consider the points of the relative Hilbert scheme only modulo the $\CS$ action and the smooth part of the fiber is a copy of $\CS$. Hence, modulo this action, we can assume that $P_2$ is supported on a fixed point on the fiber above $P_3$. We call this fixed point $P_5$. This way we get a unique representative for any point of this cycle. Now the locus of $P_1$ is the fiber above another point which is a copy of $\PO$. Therefor by projection we get a family over $\PO$. By looking at the fiber above zero and infinity of this family we get the Point-Point relation. 

\begin{figure}[h]
\centerline{\includegraphics[scale=.45]{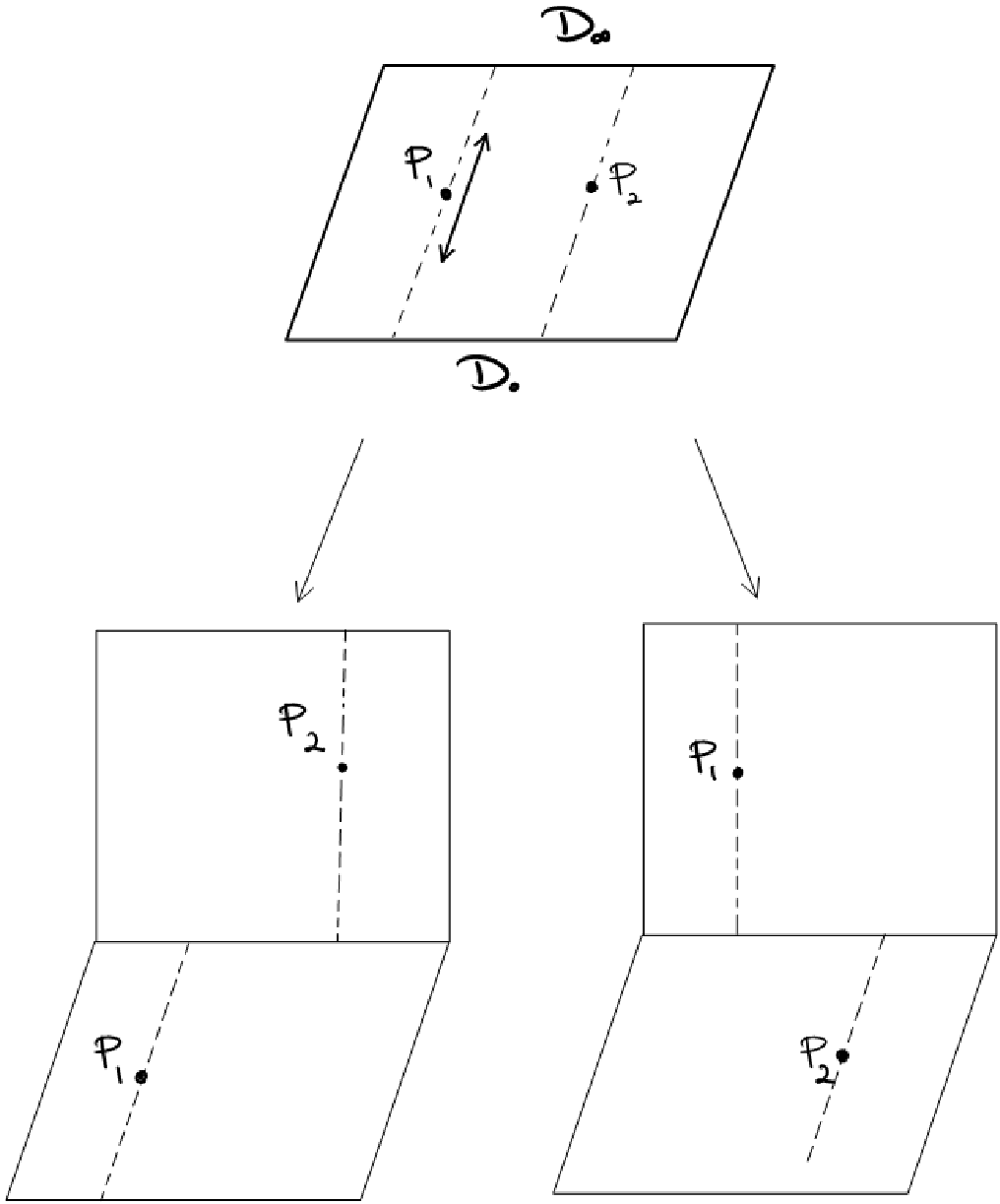}}
\end{figure}

We fix a parametrization of the fiber above $P_4$. The fiber of $\alpha_{pp}$ over $t\in \PO$ is the locus of points in $\alpha_{pp}$ where $P_1$ is supported over $t$ and $P_2$ is supported over $P_5$. Once again we use the fact that points of the relative Hilbert scheme are defined up to the $\CS$ action. The fiber of $\alpha_{pp}$ over $t\in \PO$ is also the locus of points where $P_1$ is supported over $1\in \PO$, $P_2$ is supported over $\frac{P_5}{t}$ and the locus of all the other points is multiplied by $\frac{1}{t}$. Hence in the fiber over zero $P_2$ goes to the infinity section i.e. the next bubble. Note that any point beside $P_1$ and $P_2$ is supported on the fiber over a point of $D$ or supported on the whole bubble. So when we multiply with $\frac{1}{t}$ it is supported on the same set. Hence in the fiber over zero this point might be supported on either of the $i^{th}$ bubble or the $i+1^{th}$ bubble.

The fiber above zero consists of cycles that we get by applying the following changes to cycle that we start with:

 \begin{itemize}
 \item $P_1$ will remain in the $i^{th}$ bubble. 
 \item $P_2$ will go to the $i+1^{th}$ bubble.
 \item The remaining points of the $i^{th}$ bubble might remain in that bubble or go to the next bubble.
 \item Each point with support in the bubbles with index greater than $i$ will go to the next bubble.
 \item Other points in $\PP^2 \setminus D$ or the bubbles with index less than $i$ will be in the same place.
 \end{itemize}

 The fiber above infinity is similar to the fiber above zero, with the role of $P_1$ and $P_2$ interchanged. Putting these together we arrive at the following relation:\vspace{3 mm}

   \[\displaystyle\prod_{j \in A_1, k\in A_0\setminus\{a,b\}}(\beta_0^i[k]+\beta_0^{i+1}[k])(\beta_1^i[j]+\beta_1^{i+1}[j])\beta_0^{i+1}[a]\beta_0^{i}[b].\alpha\beta^{(<i)}\beta^{(>i)+1}=\]

  \hspace{30mm}$\displaystyle\prod_{j \in A_1, k\in A_0\setminus\{a,b\}}(\beta_0^i[k]+\beta_0^{i+1}[k])(\beta_1^i[j]+\beta_1^{i+1}[j])\beta_0^i[a]\beta_0^{i+1}[b].\alpha\beta^{(<i)}\beta^{(>i)+1}$

 We will call this relation, the Point-Point relation associated to $$\alpha_{pp}=\displaystyle\prod_{j \in A_1, k\in A_0}\beta_0^i[k]\beta_1^i[j].\alpha\beta^{(<i)}\beta^{(>i)}$$ (the cycle that we start with).

\subsubsection{Point-Line Relations}

For each bubble the infinity section is a copy of $\PO$. The bubble is the projectivization of the $\mathcal{O}(1)$ over this $\PO$. This bundle has non-trivial sections, which we will fix one such section $s$. Consider the locus of points in the relative Hilbert scheme such that in the $i^{th}$ bubble there is a point $P_1$ supported on $s$, and another point $P_2$ supported on a fixed fiber of this bubble (and possibly other points in this bubble). We assume that this cycle is represented by:  $$\beta_1^{i}[a]\beta_0^{i}[b]\displaystyle\prod_{j \in A_1, k\in A_0}\beta_0^i[k]\beta_1^i[j].\alpha\beta^{(<i)}\beta^{(>i)} .$$

\begin{figure}[h]
\centerline{\includegraphics[scale=.5]{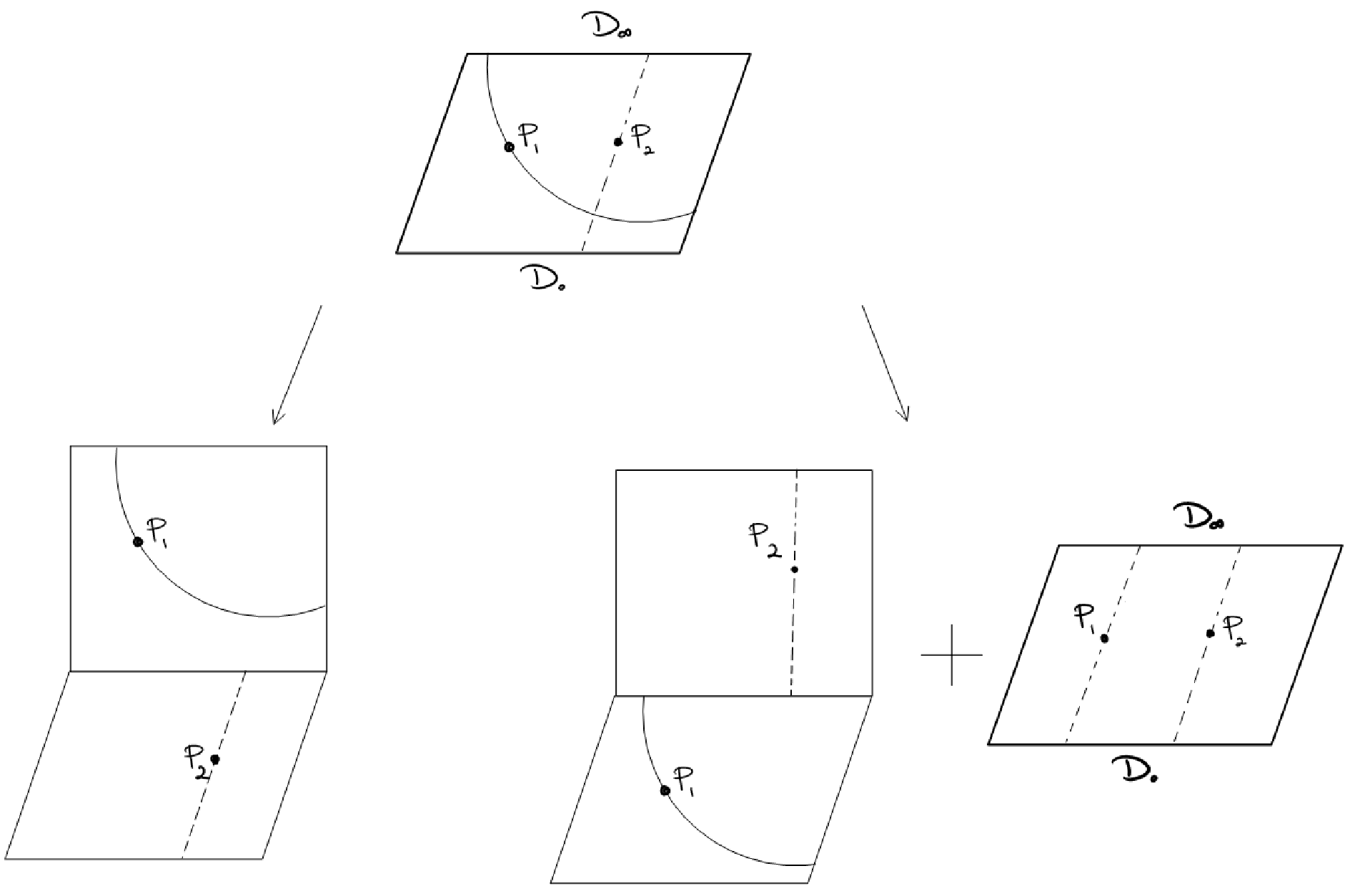}}
\end{figure}

By looking at the locus of $P_2$ we get a family over $\PO$.

When $P_2$ approaches the zero section of this bubble, by the $\CS$-action we can find a representative such that all the other points are being pushed to the infinity section.  This shows that the fiber over zero is a sum of cycles that we get by applying the following changes to the cycle that we start with:

\begin{itemize}
 \item $P_1$ will go to the $i+1^{th}$ bubble.
 \item $P_2$ will remain in the $i^{th}$ bubble.
 \item the remaining points of the $i^{th}$ bubble might remain in that bubble or go to the next bubble.
 \item Each point with support in the bubbles with index greater than $i$ will go to the next bubble.
 \item Other points in $\PP^2 \setminus D$ or the bubbles with index less than $i$ remain in the same place.
 \end{itemize}

 But the fiber over infinity is more complicated and in fact has two components. As $P_2$ goes to the infinity section either $P_1$ stays in the $i^{th}$-bubble which gives us the first component. The other possibility is for $P_1$ to go to the intersection of the infinity section and the section of $\mathcal{O}(1)$ that we fixed. In this case we get a point in the $i+1^{th}$ bubble supported on the fiber above the intersection of the fixed section and special divisor. But in this case all the remaining points should also go to the $i+1^{th}$ bubble, otherwise the resulting class will be of codimension two. So the $i^{th}$ bubble would be empty and we have to stabilize it by removing it. By this procedure we get the following relation:\vspace{3 mm}

 $\beta_1^{i+1}[a]\beta_0^{i}[b]\displaystyle\prod_{j \in A_1, k\in A_0}(\beta_0^i[k]+\beta_0^{i+1}[k])(\beta_1^i[j]+\beta_1^{i+1}[j]).\alpha\beta^{(<i)}\beta^{(>i)+1}=$

 \hspace{20mm}$\beta_1^i[a]\beta_0^{i+1}[b]\displaystyle\prod_{j \in A_1, k\in A_0}(\beta_0^i[k]+\beta_0^{i+1}[k])(\beta_1^i[j]+\beta_1^{i+1}[j]).\alpha\beta^{(<i)}\beta^{(>i)+1} + $

  \hspace{20mm}$ \beta_0^i[a]\beta_0^{i}[b]\displaystyle\prod_{j \in A_1, k\in A_0}\beta_0^i[k]\beta_1^i[j].\alpha\beta^{(<i)}\beta^{(>i)}$

  We will call this relation, the Point-Line relation associated to
  
   \[\beta_1^{i}[a]\beta_0^{i}[b]\displaystyle\prod_{j \in A_1, k\in A_0}\beta_0^i[k]\beta_1^i[j].\alpha\beta^{(<i)}\beta^{(>i)}\] .

\subsubsection{Line-line Relations}

Take a class with two points supported on one-cycles in the $i^{th}$-bubble. We represent this cycle by $$\beta_1^{i}[a]\beta_1^{i}[b]\displaystyle\prod_{j \in A_1, k\in A_0}\beta_0^i[k]\beta_1^i[j].\alpha\beta^{(<i)}\beta^{(>i)} .$$ Fix two sections $s_1$ and $s_2$ of the the $i^{th}$-bubble . Given a section of a line bundle, we can get other sections by multiplying this section by a complex number. In this way to each pair $(\lambda_1,\lambda_2)$ of non zero complex number we can associate the locus of points in the relative Hilbert scheme with two points $P_1$ and $P_2$ supported on $\lambda_1 s_1$ and $\lambda_2 s_2$ (resp), and the arrangement of the rest of the points are as in the cycle that we start with. Since we have the $\CS$ action on each bubble, the associated locus only depends on the ratio of $\lambda_1$ and $\lambda_2$. So in this way we get a family over $\CS$. Since the relative Hilbert scheme is proper we can extend this family to a family over $\PO$.

  \begin{figure}[h]
\centerline{\includegraphics[scale=.35]{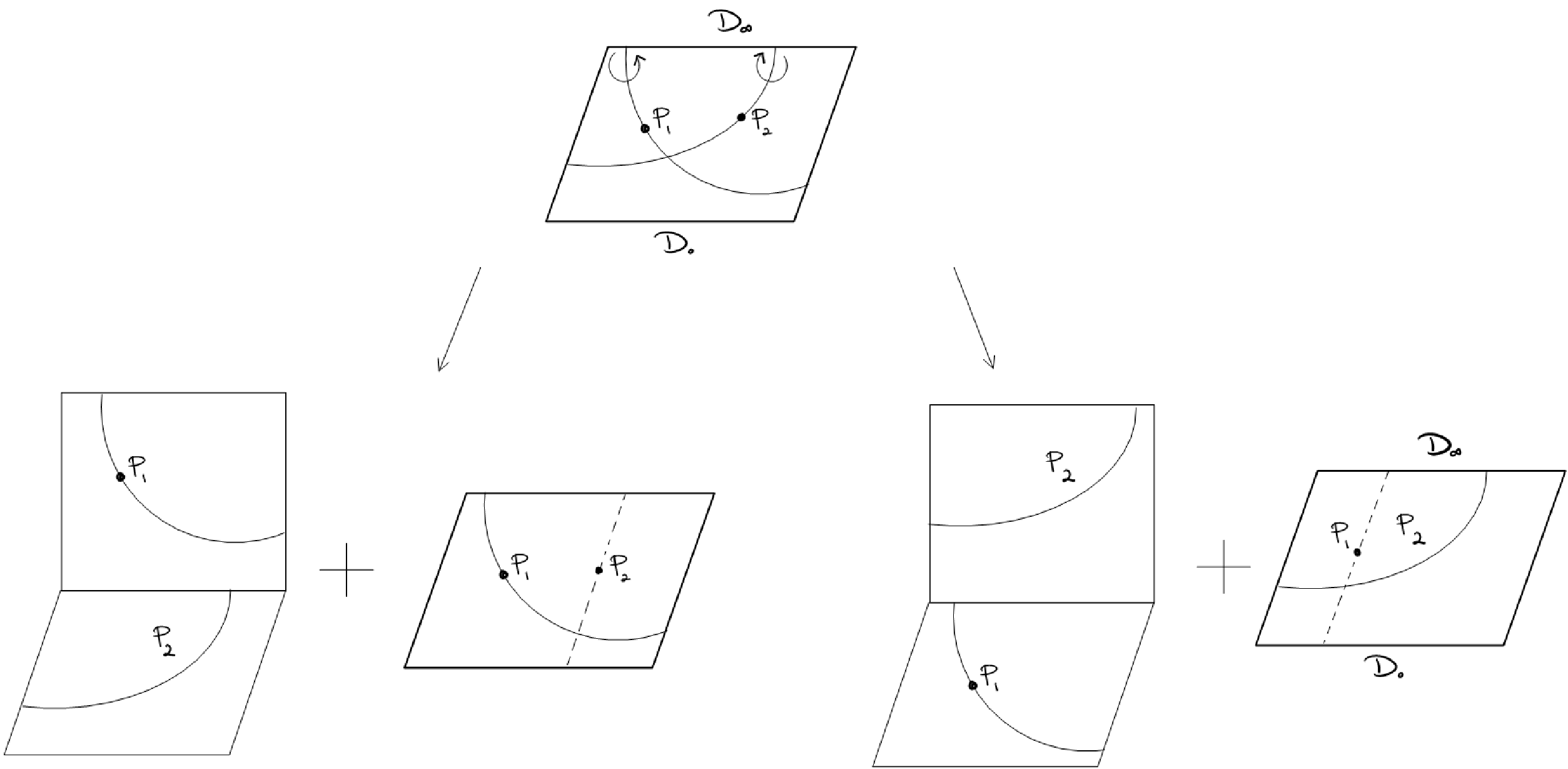}}
\end{figure}

 We are interested in the fiber above zero and infinity. The fiber above zero consist of points satisfying the following properties:
\begin{itemize}
 \item $P_1$ will go to the $i+1^{th}$ bubble, and is supported on the whole bubble.
 \item $P_2$ will remain in the $i^{th}$ bubble, and is supported on the whole bubble.
 \item the remaining points of the $i^{th}$ bubble might remain in that bubble or go to the next bubble.
 \item Each point with support in the bubbles with index greater than $i$ will go to the next bubble.
 \item Other points in $\PP^2 \setminus D$ or the bubbles with index less than $i$ will remain in the same place.
 \end{itemize}

\noindent and similarly the fiber above infinity consists of points with the same properties with the role of $P_1$ and $P_2$ interchanged. Considering this family we get the following relation:\vspace{3 mm}

 $\beta_1^{i+1}[a]\beta_1^{i}[b]\displaystyle\prod_{j \in A_1, k\in A_0}(\beta_0^i[k]+\beta_0^{i+1}[k])(\beta_1^i[j]+\beta_1^{i+1}[j]).\alpha\beta^{(<i)}\beta^{(>i)+1} + $

 $\beta_1^{i}[a]\beta_0^{i}[b]\displaystyle\prod_{j \in A_1, k\in A_0}\beta_0^i[k]\beta_1^i[j].\alpha\beta^{(<i)}\beta^{(>i)}=$

 \hspace{20mm}$\beta_1^{i}[a]\beta_1^{i+1}[b]\displaystyle\prod_{j \in A_1, k\in A_0}(\beta_0^i[k]+\beta_0^{i+1}[k])(\beta_1^i[j]+\beta_1^{i+1}[j]).\alpha\beta^{(<i)}\beta^{(>i)+1} + $

  \hspace{20mm}$\beta_0^{i}[a]\beta_1^{i}[b]\displaystyle\prod_{j \in A_1, k\in A_0}\beta_0^i[k]\beta_1^i[j].\alpha\beta^{(<i)}\beta^{(>i)}$

We will call this relation the Line-Line relation associated to 

\[\beta_1^{i}[a]\beta_1^{i}[b]\displaystyle\prod_{j \in A_1, k\in A_0}\beta_0^i[k]\beta_1^i[j].\alpha\beta^{(<i)}\beta^{(>i)}\]

So far we introduced the set of generators for the cohomology of the relative Hilbert scheme of points in the projective plane and also described four kind of relations in this group. In the following theorem we show that they are all the relations.

\begin{theorem}
The cohomology groups of ${\PP^2}_{\PO}^{[n]}$, are generated by the product cycles, and the four types of relations introduced in this section will give us a complete set of relations. 
\end{theorem}

The first part of the theorem is a consequence of Theorem ~\ref{chow-coh-sur}, and the construction of the product classes. In fact we have a stratification of the relative Hilbert scheme with pieces each isomorphic to a quotient of a product of Hilbert scheme of points on a surface. By Nakajima's theorem we know that the cohomology of these Hilbert schemes is generated by the product classes. Hence by Theorem 21 the Chow group of them is also generated by the product classes. This shows that the Chow group of the relative Hilbert scheme of points is generated by product classes. Finally by Theorem ~\ref{chow-coh-sur} the cohomology group of the relative Hilbert scheme is generated by these classes. 

In the next section we give a proof of the second part.

\section{Proof of Theorem ~\ref{relationtheorem}}

In this section we compute the dimension of space generated by the cohomology classes that we introduced in previous section modulo the relations. We deal with each type of relation separately.

\subsection{Point-Bubble Relations}

Pick a cohomology class $\alpha$ of the relative Hilbert scheme of points. Push all the points supported on cycles with dimension 0 and 1 in the $\PP^2 \setminus D$ to the bubbles. We get a representation of $\alpha$ in terms of cohomology classes with no $\alpha_0$ and $\alpha_1$ in their representation. There is one point that we want to clarify before going any further. If we start with a class that has more than one point supported on a zero cycle or a line in the $\PP^2 \setminus D$ , then we have more than one ways of writing this class in term of classes with no such points. More precisely, pick two such zero cycles $Z_1$ and $Z_2$. Pushing $Z_1$ to the bubble gives an expression (by the Point-Bubble relation), also pushing $Z_2$ to the bubble gives another expression.

\begin{lemma}\label{lemma1}
The above two expressions can be obtained from each other using other types of relations.
\end{lemma}

\begin{proof}
There are three cases that we have to consider:

1. There are two points supported on a zero-cycle in the $\PP^2 \setminus D$ with different multiplicities. If we first push one of them to the bubble and then the other we get a presentation of the original cycle in terms of cycles with fewer points supported in the $\PP^2 \setminus D$. If we push them to the bubble with a different order we get another presentation. Let $$\alpha=\displaystyle\mathop{\prod_{i\in A_2}}_{j \in A_1, k\in A_0}\alpha_2[i]\alpha_1[j]\alpha_0[k].\alpha_0[a]\alpha_0[b] \beta$$ be such a cycle, then we get the following relations:\vspace{3mm}

  $\left \{ \begin{array}{lll}
 \alpha&=&\displaystyle\mathop{\prod_{i\in A_2}}_{j \in A_1, k\in A_0}(\alpha_2[i]+\beta_1^1[i]+\beta_1^2[i])(\alpha_1[j]+\beta_0^1[j]+\beta_0^2[j])\alpha_0[k].\beta_0^1[a]\beta_0^2[b] \beta^{+2}\\
 \alpha&=&\displaystyle\mathop{\prod_{i\in A_2}}_{j \in A_1, k\in A_0}(\alpha_2[i]+\beta_1^1[i]+\beta_1^2[i])(\alpha_1[j]+\beta_0^1[j]+\beta_0^2[j])\alpha_0[k].\beta_0^2[a]\beta_0^1[b] \beta^{+2}
 \end{array} \right. $ \vspace{3mm}

Take the following class:

\[\displaystyle\mathop{\prod_{i\in A_2}}_{j \in A_1, k\in A_0}(\alpha_2[i]+\beta_1^1[i])(\alpha_1[j]+\beta_0^1[j])\alpha_0[k].\beta_0^1[a]\beta_0^1[b] \beta^{+1}\]

\noindent it has two points in the first bubble supported on zero-cycles. The above two expressions are obtained from each other using the Point-Point relation associated to this cycle.

2. There is a point supported on a zero-cycle and another point supported on a one-cycle. We get two different presentation of the cycle by pushing these points to the bubble with different orders. Let $$\alpha=\displaystyle\mathop{\prod_{i\in A_2}}_{j \in A_1, k\in A_0}\alpha_2[i]\alpha_1[j]\alpha_0[k].\alpha_1[a]\alpha_0[b] \beta$$ be the cycle that we start with. As in the previous case we get the following relations: \vspace{3mm}

$\left \{ \begin{array}{lll}
 \alpha&=&\displaystyle\mathop{\prod_{i\in A_2}}_{j \in A_1, k\in A_0}(\alpha_2[i]+\beta_1^1[i]+\beta_1^2[i])(\alpha_1[j]+\beta_0^1[j]+\beta_0^2[j])\alpha_0[k].\beta_1^1[a]\beta_0^2[b] \beta^{+2}\\
 & & + \displaystyle\mathop{\prod_{i\in A_2}}_{j \in A_1, k\in A_0}(\alpha_2[i]+\beta_1^1[i])(\alpha_1[j]+\beta_0^1[j])\alpha_0[k].\beta_0^1[a]\beta_0^1[b] \beta^{+1}\\
 \alpha&=&\displaystyle\mathop{\prod_{i\in A_2}}_{j \in A_1, k\in A_0}(\alpha_2[i]+\beta_1^1[i]+\beta_1^2[i])(\alpha_1[j]+\beta_0^1[j]+\beta_0^2[j])\alpha_0[k].\beta_1^2[a]\beta_0^1[b] \beta^{+2}
 \end{array} \right. $ \vspace{3mm}

In this case if we take this class:

\[\displaystyle\mathop{\prod_{i\in A_2}}_{j \in A_1, k\in A_0}(\alpha_2[i]+\beta_1^1[i])(\alpha_1[j]+\beta_0^1[j])\alpha_0[k].\beta_1^1[a]\beta_0^1[b] \beta^{+1}\]

\noindent the Point-Line relation associated to this cycle shows that the above two cohomology classes are equal.

3. There are two points supported on one-cycles and with different multiplicities. In the same way by pushing them to the bubble with different orders we get the following relations:\vspace{3mm}

If $$\alpha=\displaystyle\mathop{\prod_{i\in A_2}}_{j \in A_1, k\in A_0}\alpha_2[i]\alpha_1[j]\alpha_0[k].\alpha_1[a]\alpha_1[b] \beta$$

$\left \{ \begin{array}{lll}
 \alpha&=&\displaystyle\mathop{\prod_{i\in A_2}}_{j \in A_1, k\in A_0}(\alpha_2[i]+\beta_1^1[i]+\beta_1^2[i])(\alpha_1[j]+\beta_0^1[j]+\beta_0^2[j])\alpha_0[k].(\beta_1^1[a]\beta_1^2[b]+\beta_0^1[a]\beta_1^1[b]) \beta^{+2}\\
 \alpha&=&\displaystyle\mathop{\prod_{i\in A_2}}_{j \in A_1, k\in A_0}(\alpha_2[i]+\beta_1^1[i]+\beta_1^2[i])(\alpha_1[j]+\beta_0^1[j]+\beta_0^2[j])\alpha_0[k].(\beta_1^2[a]\beta_1^1[b]+\beta_1^1[a]\beta_0^1[b]) \beta^{+2}
 \end{array} \right. $ \vspace{3mm}

 In this case we take:

\[\displaystyle\mathop{\prod_{i\in A_2}}_{j \in A_1, k\in A_0}(\alpha_2[i]+\beta_1^1[i])(\alpha_1[j]+\beta_0^1[j]+)\alpha_0[k].\beta_1^1[a]\beta_1^1[b] \beta^{+1}\]

The Line-Line relation associated to this cycle shows that the above two cohomology classes are equal.
\end{proof}

\subsection{Canonical and Normal Forms}

As we discussed in the previous section given a cycle with more than one point supported on a bubble we can write down a relation in the Chow group associated to this cycle. Here we show that using these relations we can represent any cycle in terms of cycles with the \emph{canonical} form.

\begin{lemma} \label{canonical}
Each cohomology class can be represented in terms of classes satisfying the following conditions, only by using Point-Line and Line-Line relations:
\begin{itemize}
 \item There is at most one point supported on a zero-cycle in each bubble.
 \item That point (if exists) has the minimum multiplicity among the points in that bubble.
 \end{itemize}
 
 A representation of a class that satisfies these properties is called the \emph{canonical} form.
 \end{lemma}

\begin{proof}
 We start with a class $C$ that does not satisfy the above properties. Take $i$ to be the index of the first bubble that is not of that form. We recall that the Point-Line relation can be written as:\vspace{3mm}

  $\beta_0^i[a]\beta_0^{i}[b]\displaystyle\prod_{j \in A_1, k\in A_0}\beta_0^i[k]\beta_1^i[j] .\alpha\beta^{(<i)}\beta^{(>i)}=$

\hspace{20mm}$\beta_1^{i+1}[a]\beta_0^{i}[b]\displaystyle\prod_{j \in A_1, k\in A_0}(\beta_0^i[k]+\beta_0^{i+1}[k])(\beta_1^i[j]+\beta_1^{i+1}[j]).\alpha\beta^{(<i)}\beta^{(>i)+1}- $

 \hspace{20mm}$\beta_1^i[a]\beta_0^{i+1}[b]\displaystyle\prod_{j \in A_1, k\in A_0}(\beta_0^i[k]+\beta_0^{i+1}[k])(\beta_1^i[j]+\beta_1^{i+1}[j]).\alpha\beta^{(<i)}\beta^{(>i)+1} $

  \noindent all the terms in the right hand side have the same point arrangement in the first $i-1$ bubbles. Using the Point-Line relation we can write $C$ in terms of classes with fewer number of points supported on zero cycles in the $i^{th}$ bubble. So by this procedure we can fix bubbles one by one, and since the number of bubbles is finite (at most the number of points) this procedure will end at some point. Hence we can write any class in terms of classes satisfying the first condition.

 In the next step we use the Line-Line relations to write them in terms of classes that satisfy both conditions of the claim. More precisely we can write the Line-Line relation as:\vspace{3mm}
  
 $\beta_1^{i}[a]\beta_0^{i}[b]\displaystyle\prod_{j \in A_1}\beta_1^i[j].\alpha\beta^{(<i)}\beta^{(>i)}=\beta_0^{i}[a]\beta_1^{i}[b]\displaystyle\prod_{j \in A_1}\beta_1^i[j].\alpha\beta^{(<i)}\beta^{(>i)} + $

 \hspace{59mm}$\beta_1^{i}[a]\beta_1^{i+1}[b]\displaystyle\prod_{j \in A_1}(\beta_1^i[j]+\beta_1^{i+1}[j]).\alpha\beta^{(<i)}\beta^{(>i)+1} - $

 \hspace{59mm}$\beta_1^{i+1}[a]\beta_1^{i}[b]\displaystyle\prod_{j \in A_1}(\beta_1^i[j]+\beta_1^{i+1}[j]).\alpha\beta^{(<i)}\beta^{(>i)+1} .$

 It is clear from the form of these relations that the cohomology classes that we get satisfy both conditions in the first $i$ bubbles.

\end{proof}

\begin{lemma}\label{normal}
Given a class in the canonical form, we can write it as the sum of cycles with the following properties:
\begin{itemize}
 \item Each cycle is in the canonical form
 \item If there is a bubble with exactly one point supported on a zero-cycle (and no point supported on a one-cycle), then the multiplicity of that point is less than or equal to the multiplicity of any point on the next bubble. 
\end{itemize}

A representation of a class that satisfies these properties is called the \emph{normal} form.
\end{lemma}

\begin{proof}
If there is such a point in the $k^{th}$-bubble, then there are two possibilities:

1. There is a point supported on a zero-cycle in the $(k+1)^{th}$-bubble

2. All the points in the $(k+1)^{th}$-bubble are supported on one-cycles

In the first case, if this class is given by:  $$\displaystyle\beta_0^{k}[a]\beta_0^{k+1}[b]\prod_{j \in A_1}\beta_1^{k+1}[j] .\alpha\beta^{(<k)}\beta^{(>k+1)}$$ with $a>b$ then we can use the following Point-Point relation:

 $$\displaystyle(\beta_0^{k+1}[a]\beta_0^{k}[b]-\beta_0^k[a]\beta_0^{k+1}[b])\prod_{j \in A_1}(\beta_1^k[j]+\beta_1^{k+1}[j]).\alpha\beta^{(<k)}\beta^{(>k)+1}=0 .$$
 
 Which could be written as:\vspace{3mm}

$\displaystyle\beta_0^{k}[a]\beta_0^{k+1}[b]\prod_{j \in A_1}\beta_1^{k+1}[j] .\alpha\beta^{(<k)}\beta^{(>k)+1}=\displaystyle\beta_0^{k+1}[a]\beta_0^{k}[b]\prod_{j \in A_1}\beta_1^{k+1}[j] .\alpha\beta^{(<k)}\beta^{(>k)+1}+$

\hspace{30mm} $ \displaystyle(\beta_0^{k+1}[a]\beta_0^{k}[b]-\beta_0^k[a]\beta_0^{k+1}[b])\prod^{\sim}_{j \in A_1}(\beta_1^k[j]+\beta_1^{k+1}[j]).\alpha\beta^{(<k)}\beta^{(>k)+1}$

\noindent and $\displaystyle\prod^{\sim}$ means that we take every term in the expansion of the product except the one which is the multiplication of all the $\beta$ -classes in the $(k+1)^{th}$-bubble.

In the second case, if the class is given by: $$\displaystyle\beta_0^{k}[a]\beta_1^{k+1}[b]\prod_{j \in A_1}\beta_1^{k+1}[j] .\alpha\beta^{(<k)}\beta^{(>k+1)}$$ with $a>b$ then we can use the following two Point-Line relations:\vspace{3mm}

\begin{itemize}
\item $\beta_0^k[a]\beta_0^{k}[b]\displaystyle\prod_{j \in A_1}\beta_1^k[j] .\alpha\beta^{(<k)}\beta^{(>k)}=$

\hspace{30mm}$(\beta_1^{k+1}[a]\beta_0^{k}[b]-\beta_1^k[a]\beta_0^{k+1}[b])\displaystyle\prod_{j \in A_1}(\beta_1^k[j]+\beta_1^{k+1}[j]).\alpha\beta^{(<k)}\beta^{(>k)+1}$
 
\item $\beta_0^k[a]\beta_0^{k}[b]\displaystyle\prod_{j \in A_1}\beta_1^k[j] .\alpha\beta^{(<k)}\beta^{(>k)}=$

\hspace{30mm}$(\beta_0^{k}[a]\beta_1^{k+1}[b]-\beta_0^{k+1}[a]\beta_1^{k}[b])\displaystyle\prod_{j \in A_1}(\beta_1^k[j]+\beta_1^{k+1}[j]).\alpha\beta^{(<k)}\beta^{(>k)+1}$
 
\end{itemize}

If we subtract them we can write the resulting relation as: \vspace{3mm}

$\displaystyle\beta_0^{k}[a]\beta_1^{k+1}[b]\prod_{j \in A_1}\beta_1^{k+1}[j] .\alpha\beta^{(<k)}\beta^{(>k+1)}=$

\hspace{20mm}$(\beta_0^{k}[a]\beta_1^{k+1}[b])\displaystyle\prod^{\sim}_{j \in A_1}(\beta_1^k[j]+\beta_1^{k+1}[j]).\alpha\beta^{(<k)}\beta^{(>k)+1}+$

\hspace{20mm}$(\beta_1^{k+1}[a]\beta_0^{k}[b]-\beta_1^k[a]\beta_0^{k+1}[b]+\beta_0^{k+1}[a]\beta_1^{k}[b])\displaystyle\prod_{j \in A_1}(\beta_1^k[j]+\beta_1^{k+1}[j]).\alpha\beta^{(<k)}\beta^{(>k)+1}$

So in both case we can represent the original class as the sum of classes that satisfy the required properties of the statement of the lemma for the $k^{th}$-bubble. To each cycle $\alpha$ we associate the following number: \vspace{3mm}

$A(\alpha)=\sharp \left\{
\begin{array}{c|l}
          &  p \text{ is a point supported on a zero cycle with}\\
     (p,i)& \text{    no other point in that bubble }\\
          & i \text{ is the index of a bubble above point p with a}\\
          &       \text{  point in it with multiplicity less than mult(p)}
    \end{array}
 \right\}$ \vspace{3mm}

Then for both cases the representation that we get consist of classes with smaller $A$. Since the number of all these classes is finite this procedure ends at some point. Hence any cycle in the resulting presentation satisfies the properties of the lemma.

\end{proof}

\subsection{Computation of Betti Numbers}

In order to compute the dimension of the cohomology groups of the relative Hilbert scheme of points, we start by counting the number of cycles of a given dimension in the canonical form. 

Let $a_{u,v}$ be the number of cycles of the form $\displaystyle\prod_{j \in A_1}\beta_1^i[j]$ with $v$ points and dimension $u$ i.e. $\displaystyle\sum_{j \in A_1\neq \emptyset} 2j+2=u+2$. Each $\beta_1^i[j]$ adds $j$ to the number of points of  this cycle and adds $2j+2$ to the dimension of the cycle. We claim that $$\displaystyle\sum_{u,v}a_{u,v}t^uq'^v=1+\frac{1}{t^{2}}\left(-1+\prod_{m=1}^{\infty}\frac{1}{1-t^2(t^2q')^{m}}\right) .$$

To see this, note that any term in the expansion of $$\displaystyle\prod_{m=1}^{\infty}\frac{1}{1-t^2(t^2q')^{m}}$$ is obtained as follows. Pick a set of integers $n_1,\cdots,n_k$ and correspondingly  consider the contribution of $$\displaystyle\prod_{l=1}^{k}\frac{1}{1-t^2(t^2q')^{n_l}}.$$ Then for each $l$ in $\{1,\cdots,k\}$ choose $m_l$ which ought to be the multiplicity of $n_l$. Correspondingly we get the term $\displaystyle\prod_{l=1}^{k} {\left( t^2(t^2q')^{n_l}\right) }^{m_l}$ in the expansion. To this term we associate the cycle $\displaystyle\prod_{l=1}^{k}\left( \beta_1^i[n_l]\right)^{m_l}$. This is a bijection, so the generating function for $a_{u,v}$ is $$\displaystyle 1+\frac{1}{t^{2}}\left(-1+\prod_{m=1}^{\infty}\frac{1}{1-t^2(t^2q')^{m}}\right) .$$ Note that the term $\displaystyle\frac{1}{t^2}$ reflects the fact that we take the quotient with $\CS$ which subtracts $2$ from the dimension.

Since we are interested in the normalized Poincar\'{e} polynomial, we make the change of variable $q=t^2q'$. Therefore the generating function becomes: $$1+\frac{1}{t^{2}}\left(-1+\displaystyle\prod_{m=1}^{\infty}\frac{1}{1-t^2q^{m}}\right).$$

In order to count the number of possible configurations in the $i^{th}$ bubble we have to allow the cycle to have one point supported on a zero-cycle i.e. a $\beta_0^i$ term. In the canonical form this point has the least multiplicity among the points in the same bubble. So the number of cycles of the form $\beta_0[k]\displaystyle\prod_{j \in A}\beta_1^i[j]$ with dimension $u$ and $v$ points is equal to the number of cycles of the form $\displaystyle\prod_{j \in B}\beta_1^i[j]$ with dimension $u-2$ and $v$ points. The correspondence is given by sending $\beta_0[k]\displaystyle\prod_{j \in A}\beta_1^i[j]$ to $\beta_1[k]\displaystyle\prod_{j \in A}\beta_1^i[j]$. The inverse map on $C=\displaystyle\prod_{j \in B}\beta_1^i[j]$ is given as follows. Choose $j_0$ to be the minimum of $j \in B$. Send $C$ to $\beta_0[j_0]\displaystyle\prod_{ j \in B\setminus \{j_0\}}\beta_1^i[j]$. Hence the generating function for the number of possible configurations of a bubble in the canonical form is given by:

$$1+\frac{1+t^{-2}}{t^{2}}\left(-1+\displaystyle\prod_{m=1}^{\infty}\frac{1}{1-t^2q^{m}}\right).$$

Since in $\PP^2 \setminus D$ we have only points supported on two-cycles (i.e. $\alpha_2$ terms), the generating function for the number of cycles in the canonical form is given by:

$$\displaystyle\prod_{m=1}^{\infty}\frac{1}{1-t^{2}q^{m}} \sum_{i=0}^{\infty}\left(\frac{1+t^{-2}}{t^{2}}\left(-1+\displaystyle\prod_{m=1}^{\infty}\frac{1}{1-t^2q^{m}}\right)\right)^{i} .$$

Each cycle in the canonical form which does not satisfy the properties of Lemma ~\ref{normal} (i.e. is not in the normal form) has at least one point supported on a zero-cycle in a bubble with multiplicity greater than the minimum multiplicity in the next bubble. We call such a point a bad point and such a cycle a bad cycle. 

\begin{definition}
Let $C$ be a cycle in the canonical form. We assume that 
$$\displaystyle C=\alpha \beta^{< i}\beta^{> i} \beta^{i}_{0 or 1}[n_1]\prod_{l=1}^{k}(\beta_1^i[n_l])^{m_l}$$ with $n_1 < \cdots < n_k $. A \emph{marking} on $C$ in the $i^{th}$ bubble is a subset $S=\{s_1,\cdots,s_t\}$ of $\{2,\cdots,k\}$ and we denote it by $(C,S)$. We assume that $s_1 > \cdots > s_t$. 

To a marked cycle $(C,S)$ we associate the following cycle: 
$$C_S:=(-1)^t \alpha \beta^{< i}\beta^{(> i)+t}\left[\beta^{i+t}_{0 or 1}[n_1]\prod_{l=1}^{k}(\beta_1^{i+t}[n_l])^{\widehat{m_l}} \right] \prod_{j=1}^{t} \beta_0^{i+1-j}[n_j]$$ in which 
$$\widehat{m_l}= \begin{cases} m_l &\mbox{if } l\notin S \\ 
m_l-1 & \mbox{if } l\in S \end{cases}.$$ 

Note that  $dim(C_S)=dim(C)-4t$. A marking for $C$ is a choice of marking for each bubble (we also consider the empty set as a marking).
\end{definition}

For example if $$C=\begin{array}{l}
\beta_1^3[n_5]\\
\beta_1^2[n_1] (\beta_1^2[n_2])^2 \beta_1^2[n_3] \beta_1^2[n_4] \\ 
\beta_1^1[n_0]
\end{array},$$ (i.e. $C$ is the product of expressions in all rows and each row denotes the term in one bubble) with marking $S=\{2,4 \}$ we have: 
$$C_S=\begin{array}{l}
\beta_1^5[n_5] \\
\beta_1^4[n_1] \beta_1^4[n_2] \beta_1^4[n_3] \\
\beta_0^3[n_2] \\
\beta_0^2[n_4] \\
\beta_1^1[n_0] 
\end{array} .$$

For a cycle $D$ in the canonical form, let $\mathcal{M}_D$ be the set of all marked cycles $(C,S)$ so that $C_S=\pm D$. 
 
\begin{lemma}
$\displaystyle \sum_{(C,S)\in \mathcal{M}_D} C_S = \begin{cases} D &\mbox{if } D \text{ is in normal form} \\ 
0 & \mbox{otherwise } \end{cases}$.
\end{lemma}

\begin{proof}
First note that if $S$ is non empty then by construction $C_S$ is a bad cycle. Hence if $D$ is in normal form then the only element of $\mathcal{M}_D$ is $(D,\emptyset)$ and the result follows. 

Let $D$ be a bad cycle. Thus $D$ has a bad point with multiplicity $n_1$ and we assume that $D$ is given by:  $$D=\alpha \beta^{< i}\beta^{> i+d} \beta^{i+d}_{0 or 1}[n_{min}]\prod_{j\in A}\beta_1^{i+d}[j] \prod_{k=1}^{d} \beta_0^{i+k-1}[n_k]$$
\noindent in which $n_1 < n_2 < \cdots < n_d$ and $d$ is chosen to be the largest number with such a sequence of bad points starting from $n_1$. We call the set $\{n_1 ,n_2, \cdots, n_d \}$ a bad component of $D$. 

 Given a subset $T=\{t_1,\cdots,t_c\}$ of $\{ 1, \cdots,d \}$ with $t_1 > \cdots > t_c$, we take $t_0:=d+1$ and define $D^T$ to be the following cycle with every element of $T^c$ as marked.
 $$\begin{array}{r}
 D^T:=\alpha \beta^{< i}\beta^{(> i+d)-d+c} \left(\beta^{i+c}_{0 or 1}[n_{min}]\prod_{j\in A}\beta_1^{i+c}[j] \prod_{t_c<k\leq d}\beta_1^{i+c}[n_k] \right). \\
 \left[\prod_{a=1}^{c}  \beta_0^{i+a-1}[n_{t_a}]\left( \prod_{t_{a-1}<k<t_a} \beta_1^{i+a-1}[n_k]  \right)\right] .\end{array} $$ 
 
For example let $D=\begin{array}{l}
\beta_{1}^{5}[1] \beta_{1}^{5}[6] \\
\beta_{0}^{4}[3] \\
\beta_{0}^{3}[5] \\
\beta_{0}^{2}[7] \\
\beta_{0}^{1}[9] \\
\end{array}$ so $n_1=3$ , $d=4$. Let $T=\{2,4\}$, then $D^T=\begin{array}{l}
\beta_{1}^{3}[1] \beta_{1}^{3}[6] \beta_{1}^{3}[3]\\
\beta_{0}^{2}[5] \beta_{1}^{2}[7]\\
\beta_{0}^{1}[9] \\
\end{array}$. If $T=\{ 1, \cdots,d \}$, $D=D^T$. 

We consider the case in which $D$ has only one bad component. In this case $\mathcal{M}_D$ is the set $(D^T,T^c)$. Hence $$\displaystyle \sum_{(C,S)\in \mathcal{M}_D} C_S = \sum_{T \subset \{1,\cdots,d\}} (D^T)_T = \left(\sum_{b=0}^{d} (-1)^d {d \choose b}\right) . D = 0. $$

 If $D$ has multiple bad components then the sum breaks as the product of the associated sums for each component, hence it is zero.

\end{proof}

By our argument in the beginning of this subsection we have a correspondence between the terms in $$1+\frac{1+t^{-2}}{t^{2}}\left(-1+\displaystyle\prod_{m=1}^{\infty}\frac{1}{1-t^2q^{m}}\right)$$ and the canonical cycles. Since $dim(C_S)=dim(S)-4|S|$, if we change the generating function by replacing each term $\displaystyle\prod_{l=1}^{k} {\left( t^2(t^2q')^{n_l}\right) }^{m_l}$ with $$\displaystyle\prod_{l=1}^{k} {\left( t^2(t^2q')^{n_l}\right) }^{m_l} (1-t^{-4})^{k-1}= \prod_{l=1}^{k} {\left( t^2(t^2q')^{n_l}\right) }^{m_l}\sum_{T\subset \{1,\cdots,k\}} (-1)^{|S|} t^{-4 |S|} ,$$ then we obtain the generating function for $\displaystyle \sum_{(C,S)\in \mathcal{M}_D} C_S$. By the above lemma the aforementioned generating function is the same as the generating function for cycles in  normal form.

 Hence, if we consider $$\displaystyle\frac{1+t^{-2}}{t^{2}}\left(-1+\prod_{m=1}^{\infty}\left((\frac{1}{1-t^2q^{m}}-1)(1-t^{-4})+1\right)\right),$$ which corresponds to taking terms of the form $\displaystyle\prod_{l=1}^{k} {\left( t^2(t^2q')^{n_l}\right) }^{m_l} (1-t^{-4})^{k}$ it differs from the contribution of one bubble to the generating function for normal forms by a factor of $\frac{1}{(1-t^{-4})}$. Thus this contribution is given by:

$$\displaystyle\frac{1+t^{-2}}{t^{2}(1-t^{-4})}\left(-1+\displaystyle\prod_{m=1}^{\infty}\left((\frac{1}{1-t^2q^{m}}-1)(1-t^{-4})+1\right)\right)+1 $$

Since $\displaystyle\frac{1+t^{-2}}{t^2(1-t^{-4})}=\frac{1}{t^2-1}$ the above generating function takes the form:

$$1+\frac{1}{t^{2}-1}\left(-1+\displaystyle\prod_{m=1}^{\infty}\left[\left(\frac{1}{1-t^2q^{m}}-1\right)(1-t^{-4})+1\right]\right)= 1+\frac{1}{t^{2}-1}\left(-1+\displaystyle\prod_{m=1}^{\infty}\frac{1-t^{-2}q^{m}}{1-t^2q^{m}}\right).$$

Hence the generating function for the cycles in the normal form is given by:  

\[\begin{array}{lll}
\displaystyle\prod_{m=1}^{\infty}\frac{1}{1-t^{2}q^{m}} \sum_{i=0}^{\infty}\left(\frac{-1+\prod_{m=1}^{\infty}\frac{1-t^{-2}q^{m}}{1-t^2q^{m}}}{t^{2}-1}\right)^{i} &=& \displaystyle\prod_{m=1}^{\infty}\frac{1}{1-t^{2}q^{m}} \frac{1}{1-\frac{-1+\prod_{m=1}^{\infty}\frac{1-t^{-2}q^{m}}{1-t^2q^{m}}}{t^{2}-1}} \\
 &=& \displaystyle\frac{t^2-1}{t^{2}\prod_{m=1}^{\infty}(1-t^2q^{m})-\prod_{m=1}^{\infty}(1-t^{-2}q^{m})}
\end{array}\]

This generating function agrees with the result of example~\ref{plane example}. Hence the space of cycles in normal form has the same dimension as the cohomology group of the relative Hilbert space. Therefore the cohomology group is isomorphic to the space of normal forms. In particular the relations in Theorem ~\ref{relationtheorem} are all the relations. This completes the proof of Theorem ~\ref{relationtheorem}.

\section{Acknowledgements}

I am very gratful to my advisor Rahul Pandharipande for sharing these ideas with me and many related discussion, and contributing Theorem 1. I would like to thank E. Eftekhary, A. Oblomkov and V. Shende for many related discussions.


\end{document}